\documentclass[11pt]{amsart}

\usepackage{amssymb,latexsym}

\usepackage{graphicx,epsfig}

\usepackage[dvipsnames]{xcolor}

\def\cal{\mathcal}

\def\Bbb{\mathbb}

\def\N{{\Bbb N}}
\def\Z{{\Bbb Z}}
\def\R{{\Bbb R}}

\def\cN{{\cal N}}

\def\F{{\cal F}}	
\def\ext{{\cal E}} 

\def\cS{{\cal S}}		
\def\I{{\cal I}}
\def\J{{\cal J}}
\def\cR{{\cal R}}

 \def\om{{\omega}}

\def\de{{\delta}}
\def\la{{\lambda}}

\def \dist {\text{\rm dist\,}}
\def \supp {\text{\rm supp\,}}

\def\ve{\varepsilon}
\renewcommand{\epsilon}{\varepsilon}
\def\trans{{\,}^t}

\def\pa{{\partial}}

\def\al{\alpha}
\def\be{\beta}
\def\ga{\gamma}

\def\cal{\mathcal}

\def\Bbb{\mathbb}

\def\Hyp{\rm{Hyp}}

\def\E{{\mathcal E_\phi}}
\def\T{{\mathbb T}}

\def\muk{\mu^{1/2}K^{-3/4}}

\newcommand{\floor}[1]{\left\lfloor #1 \right\rfloor}

\newtheorem{thm}{Theorem}[section]

\newtheorem{proposition}[thm]{Proposition}
\newtheorem{cor}[thm]{Corollary}
\newtheorem{lemma}[thm]{Lemma}
\newtheorem{remark}[thm]{Remark}
\newtheorem{remarks}[thm]{Remarks}
\newtheorem{defn}[thm]{Definition}

\newtheorem{defns}[thm]{Definitions}

\begin{document}

\title[ Fourier restriction for smooth hyperbolic 2-surfaces]{Fourier restriction for smooth hyperbolic 2-surfaces}

\author[S. Buschenhenke]{Stefan Buschenhenke}
\address{S. Buschenhenke:  Mathematisches Seminar, C.A.-Universit\"at Kiel,
Ludewig-Meyn-Stra\ss{}e 4, D-24118 Kiel, Germany}
\email{{\tt buschenhenke@math.uni-kiel.de}}
\urladdr{http://www.math.uni-kiel.de/analysis/de/buschenhenke}

\author[D. M\"uller]{Detlef M\"uller}
\address{D. M\"uller: Mathematisches Seminar, C.A.-Universit\"at Kiel,
Ludewig-Meyn-Stra\ss{}e 4, D-24118 Kiel, Germany}
\email{{\tt mueller@math.uni-kiel.de}}
\urladdr{http://www.math.uni-kiel.de/analysis/de/mueller}

\author[A. Vargas]{Ana Vargas}
\address{A. Vargas: Departmento de Mathem\'aticas, Universidad Aut\'onoma de  Madrid,
28049
Madrid,
Spain
}
\email{{\tt ana.vargas@uam.es}}
\urladdr{{http://matematicas.uam.es/~AFA/}}

\thanks{The first two authors were partially supported by the DFG grant MU 761/
11-2.\\
The third author was partially supported by grants PID2019-105599GB-I00 and
MTM2016-76566-P (Ministerio de Ciencia, Innovaci$\acute{\text{o}}$n y Universidades), Spain.}

\begin{abstract} 
We prove Fourier restriction estimates by means of the polynomial partitioning method  for compact subsets of any   sufficiently smooth hyperbolic
hypersurface in $\R^3.$ Our approach exploits in a crucial way the underlying hyperbolic geometry,  which leads to a novel notion of strong transversality and corresponding  ``exceptional'' sets. For the division of these exceptional  sets we make crucial and perhaps surprising use of a lemma on level sets  for sufficiently smooth one-variate functions from a previous article of ours. 

\end{abstract}

\maketitle


\tableofcontents

\thispagestyle{empty}
\setcounter{equation}{0}
\section{Introduction}\label{intro}
 
Let $S\subset \R^n$ be a sufficiently smooth hypersurface. The Fourier restriction problem, introduced by E. M. Stein in the seventies (for general submanifolds), asks for  the range  of exponents
$\tilde p$ and $\tilde q$ for which  an a priori  estimate of the form
\begin{align*}
\bigg(\int_{S}|\widehat{f}|^{\tilde q}\,d\sigma\bigg)^{1/\tilde q}\le C\|f\|_{L^{\tilde
p}(\R^n)}
\end{align*}
holds  true for   every Schwartz function   $f\in\mathcal S(\R^n),$ with a constant $C$
independent of $f.$ Here, $d\sigma$ denotes the Riemannian surface measure on $S.$

The sharp range in dimension  $n=2$ for curves with non-vanishing curvature was determined
through work by  C. Fefferman, E. M. Stein and A. Zygmund \cite{F1}, \cite{Z}. In higher
dimension, the sharp $L^{\tilde p}-L^2$ result  for hypersurfaces with
non-vanishing Gaussian curvature was obtained by E. M. Stein and P. A. Tomas \cite{To},
\cite{St1} (see also Strichartz \cite{Str}). Some more general classes of surfaces were
treated by A. Greenleaf \cite{Gr}. In work by I. Ikromov, M.
Kempe and D. M\"uller  \cite{ikm}  and Ikromov and M\"uller \cite{IM-uniform}, \cite{IM},
the sharp range of Stein-Tomas type    $L^{\tilde p}-L^2$  restriction
estimates has been  determined  for a large class  of smooth, finite-type hypersurfaces,
including all analytic hypersurfaces.

The question about  general  $L^{\tilde p}-L^{\tilde q}$ restriction
estimates is nevertheless still wide open. Fourier restriction to hypersurfaces  with non-negative principal curvatures has been studied intensively by many authors. Major progress was due to J. Bourgain in the nineties (\cite{Bo1}, \cite{Bo2}, \cite{Bo3}). At the end of that decade the bilinear method was introduced (\cite{MVV1}, \cite{MVV2}, \cite{TVV} \cite{TV1}, \cite{TV2},  \cite{W2}, \cite{T2},  \cite{lv10}). A new impulse to the problem has been given with  the multilinear method (\cite{BCT}, \cite{BoG}). The best results up to date have been obtained with the polynomial partitioning method, developed by L. Guth (\cite{Gu16},
\cite{Gu18}) (see also \cite{hr19} and \cite{WA18} for recent improvements).

\smallskip
For the case of hypersurfaces of  non-vanishing  Gaussian curvature but principal curvatures of different signs, besides Tomas-Stein type Fourier restriction estimates, until a few years ago the only case which had been studied successfully was the case of the hyperbolic paraboloid (or ``saddle") in  $\R^3$:
in 2005, independently S. Lee \cite{lee05} and A. Vargas \cite{v05}  established results analogous to Tao's theorem \cite{T2} on elliptic surfaces (such as the $2$ -sphere), with the exception of the  end-point, by means of the bilinear method. 

\smallskip
First results based on the bilinear approach for particular  one-variate perturbations of the saddle were eventually proved by the authors in \cite{bmv20}, \cite{bmvp19}  and \cite{bmvp20a}. Furthermore, B. Stovall \cite{Sto17b}  was able to include also the end-point case for the hyperbolic paraboloid.  Building on the papers \cite{lee05}, \cite{v05} and   \cite{Sto17b}, and by  strongly making use of Lorentzian symmetries,  even  global restriction estimates for one-sheeted hyperboloids have been established  recently by B. Bruce, D. Oliveira e Silva and B. Stovall \cite{bros20}, with extensions to higher dimensions by Bruce  \cite{br20b}. Results on higher dimensional hyperbolic paraboloids  have been reported  by A. Barron \cite{ba20}.  All these results are in the bilinear range given by \cite{T2}.

\smallskip

Improvements over the results  for the saddle  by means of an adaptation  of the polynomial partitioning method from  Guth's articles   \cite{Gu16} were achieved   by C. H. Cho and J. Lee \cite{chl17}, and J. Kim \cite{k17}.  
  Moreover, for a particular class of  one-variate perturbations of the hyperbolic paraboloid, an analogue of Guth's result had been proved by the authors in \cite{bmvp20b}, and more lately  were reported by S. Guo and C. Oh  for compact subsets of general polynomial surfaces with negative Gaussian curvature in \cite{go20}. Moreover, by making use of Lorentzian symmetries, B.  Bruce \cite{br20a} has established analogous results for compact subsets of the one-sheeted  hyperboloid.
\smallskip

\smallskip
In this article, we shall obtain the analogous result to \cite{Gu16} for compact subsets of  any  sufficiently smooth  hyperbolic surface.

\smallskip
More precisely, we shall  study embedded $C^m$- hypersurfaces  $S$ in $\R^3$ of sufficiently high degree of  regularity $m\ge 3$  which are {\it  hyperbolic} in the sense that the Gaussian curvature is strictly negative at every point, i.e., that at every point of $S$ one principal curvature is strictly positive, and the other one  is strictly negative.

\smallskip
As usual, it will be more convenient to use duality and work in the adjoint setting.  If
$\cR$ denotes the Fourier restriction operator $g\mapsto \cR g:=\hat g|_{S}$ to the surface
$S,$ its adjoint operator $\cR^*$ is given by $\cR^*f(\xi)=\ext f(-\xi),$ where
 $\ext=\ext_{S}$ denotes the ``Fourier extension'' operator given by
\begin{equation*}
\ext f(\xi):=\widehat{f\,d\sigma}(\xi)= \int_{S} f(x)e^{-i\xi\cdot x}\,d\sigma(x),
\end{equation*}
with  $f\in L^q(S,\sigma).$ The restriction problem is therefore
equivalent to the question of finding the appropriate range of exponents for which the
estimate
$$
\|\mathcal E f\|_{L^p(\R^3)}\le C\|f\|_{L^q(S,d\sigma)}
$$
holds true with a constant $C$ independent of the function   $f\in L^q(S,d\sigma).$ We shall here concentrate on local estimates of this form, where $S$ is replaced by a sufficiently small neighborhood of  a given point on $S.$ Such estimates then allow for estimates of the form
\begin{equation}\label{localefr}
\|\mathcal E_{S_c} f\|_{L^p(\R^3)}\le C_{S_c,p,q}\|f\|_{L^q(S,d\sigma)},
\end{equation}
for any compact subset $S_c$ of $S,$ where we have put
\begin{equation*}
\ext_{S_c} f(\xi):= \int_{S_c} f(x)e^{-i\xi\cdot x}\,d\sigma(x).
\end{equation*}
Our main result will be
\begin{thm}\label{mainresult1}
Assume  that $p>3.25$ and  $p>2q'.$  Then there is some sufficiently large $M(p,q)\in \N$ such that for any embedded hyperbolic hypersurface $S\subset \R^3$ of class $C^{M(p,q)}$ the estimate  \eqref{localefr} holds true, i.e., for any compact subset $S_c$ of $S,$ we have
$$
\|\mathcal E_{S_c} f\|_{L^p(\R^3)}\le C_{S_c,p,q}\|f\|_{L^q(S,d\sigma)}.
$$
\end{thm}
\smallskip

For the proof of this result, we shall consider the following classes of functions:
Let  $\Sigma:=[-1,1]\times [-1,1].$  For $M\in\N, M\ge 3,$ we denote by  $\Hyp^M=\Hyp^M(\Sigma)\subset  C^M(\Sigma)$ the  set of all  functions $\phi$    on $\Sigma$ satisfying the following properties:

 $\phi$ extends from $\Sigma$ to a $C^M$-function on $2\Sigma,$ also denoted by  $\phi,$  which satisfies the following   conditions \eqref{normal1},\eqref{regularity} on $2\Sigma:$ 
\begin{equation}\label{normal1}
\phi(0)=0, \ \nabla\phi(0)=0, \
D^2\phi(0)=\left(
\begin{array}{ccc}
 0 &   1  \\
 1 &  0   \\
\end{array}
\right),\\
\end{equation}
and
\begin{equation}\label{regularity}
\|\pa_x^\al\pa_y^\beta\phi\|_\infty \le 10^{-5} \quad \text {for} \quad 3\le \al+\beta\le M.
\end{equation}

\noindent {\bf Note: } (i)  $\phi\in C^M(\Sigma)$ lies in $\Hyp^M$ if and only if $\phi(x,y)$ is a small perturbation of $xy,$ in the sense that $\phi(x,y)=xy+\psi(x,y),$
where $\psi(0)=0, \ \nabla\psi(0)=0,$ and $D^2\psi(0)=0,$ and $\psi$ satisfies the estimates \eqref{regularity} (even on $2\Sigma$), with $\phi$ replaced by $\psi.$

(ii) If  $\phi\in\Hyp^M,$ then
\begin{equation}\label{phiest}
\max\big\{\|\phi_{xx}\|_\infty,  \|\phi_{yy}\|_\infty, \|\phi_{xy}-1\|_\infty\big\}\le 2\cdot 10^{-5}.
\end{equation}

Our key result then is
\begin{thm}\label{mainresult2}
Assume  that $p>3.25$ and  $p>2q'.$  Then there is some sufficiently large $M(p,q)\in \N$ such that for any $\phi\in  \Hyp^{M(p,q)}$  the Fourier extensions operator
$$
\ext _\phi f(\xi):=\int_{\Sigma} f(x,y) e^{-i(\xi_1 x+\xi_2 y+\xi_3\phi(x,y))}  \, dx dy
$$
associated to the graph $S_\phi$ of $\phi$ satisfies the estimate
\begin{equation}\label{mainest}
\|\ext_\phi f\|_{L^p(\R^3)}\le C_{p,q}\,\|f\|_{L^q(\Sigma)} \qquad \text {for every} \quad f\in \cS(\R^2),
\end{equation}
with a constant which is independent of $\phi$ and $f.$
\end{thm}

Note that Theorem \ref{mainresult1} follows easily from Theorem \ref{mainresult2}. Indeed, if  $S$ is  an embedded hyperbolic hypersurface $S\subset \R^3$ of class $C^{M(p,q)},$ with $M(p,q)$ as in Theorem \ref{mainresult2}, and if $S_c$ is a compact subset, then by compactness we can localized to sufficiently small neighborhoods of any points $x^0$ in $S_c.$ So, after permuting coordinates, we may assume that near such a point $S$ is  given as the graph of a $C^{M(p,q)}$ function $\phi,$ and after translation and linear change of coordinates, that $x^0=0\,$ and that $S$ is the graph of $\phi$ over some sufficiently small neighborhood $U$ of the origin in $\R^2,$  where $\phi$ satisfies \eqref{normal1}. Finally, after applying a suitable isotropic scaling   by putting
$\tilde \phi(z):=\frac 1{r^2} \phi(rz),$ where $0<r\ll 1,$ assuming that $U$ was sufficiently small, we see that we can reduce to a function $\tilde \phi$ in
$\Hyp^{M(p,q)}(\Sigma).$
\smallskip

 Denote by $B_R$ the cube $B_R:=[-R,R]^3, \, R\ge 0.$ 
 In a similar way as in \cite{bmvp20b}, Theorem \ref{mainresult2} will be a consequence of the following local Fourier extension estimate:

\begin{thm}\label{mainresult}
Assume that  $3.25\geq q >2.6.$Then, for every $\epsilon>0,$  there is a sufficiently large $M(\epsilon)\in \N$ such that for any $\phi\in  \Hyp^{M(\epsilon)}$ the following holds true:   there is a constant $C_\epsilon$  such that for any  $R\ge1$
\begin{equation}\label{mainresultest}
\|\E f\|_{L^{3.25}(B_R)} \leq C_\epsilon R^\epsilon \|f\|_{L^2(\Sigma )}^{2/q}\, \|f\|_{L^\infty(\Sigma)}^{1-2/q},
\end{equation}
		
for all $f\in L^\infty(\Sigma)$.
\end{thm}

Indeed,  a simple interpolation argument as in \cite{bmvp20b} shows that the estimate in Theorem \ref{mainresult} implies the following:

If   $p>3.25$ and  $p>2q',$ then 
\begin{equation}\label{mainestR}
\|\ext_\phi f\|_{L^p(B_R)}\le C_{p,q,\epsilon} R^\epsilon\,\|f\|_{L^q(\Sigma)},
\end{equation}
with a constant which is independent of $\phi$ and $f.$ 

Finally,  as usual, we can invoke an  $\epsilon$-removal theorem to pass to Theorem \ref{mainresult2}, but we have to be a bit more precise here than usually: 

From  \cite [Theorem 5.3]{k17} (which, as Kim  observes, is an immediate extension of Tao's $\epsilon$-removal theorem  \cite [Theorem 1.2]{T99}),  applied to the adjoints  of the restriction operators,
 we see that  the estimate in \eqref{mainestR} implies  the following: there is a constant $C>0$ such that 
if 
\begin{equation}\label{eremov1}
\frac 1 p>\frac 1{\tilde p}+\frac {C}{-\log \epsilon},
\end{equation}
then 
\begin{equation}\label{mainest4}
\|\ext_\phi f\|_{L^{\tilde p}(\R^3)}\le C_{ \tilde p,q}\,\|f\|_{L^q(\Sigma)}.
\end{equation}
Thus, given $\tilde p$ with    $\tilde p>3.25$ and  $\tilde p>2q',$ we can first choose an appropriate $p$ such that  $\tilde p>p>3.25$ and  $\tilde p>p>2q',$ and then an appropriate $\epsilon=\epsilon(\tilde p,q)>0$ so that \eqref{eremov1} holds true, hence also  \eqref{mainest4}. This proves Theorem \ref{mainresult2}.

\smallskip 
Later in our proofs we shall always assume without further mentioning that $\ve$ is sufficiently small.

\bigskip

\noindent\textsc{Outline of the paper:} 
Since Guth's polynomial partitioning method has been discussed in various papers by now, we shall  be brief in many parts and refer to Guth \cite{Gu16} and our previous paper \cite{bmvp20b} whenever possible. Instead, we will focus on the novelties of our approach, so that some familiarity  with the polynomial partitioning method is recommend.
\smallskip

A crucial property in restriction estimates is ``strong''  transversality (which we shall define precisely at a later stage). A major task in our paper will then be to understand the "exceptional" sets which  lack this kind of transversality.  In case of the unperturbed hyperbolic paraboloid, i.e., the graph of $xy,$   any two small caps that are not strongly transversal have to be contained in the part of the surface lying over a vertical or horizontal strip in the $(x,y)$-plane. For our general hyperbolic surfaces,  the "exceptional" sets are again certain rectangles, but their size and slope strongly depending on the geometry of the surface.  Rescaling arguments for functions restricted to a rectangle are  here a priori difficult, since, unlike for the unperturbed hyperbolic paraboloid, our class of phase functions $\Hyp^m$ is not closed under anisotropic scalings, nor any other suitably large group of symmetries, such as the Lorentz group in case of the one-sheeted hyperboloid.  However, we will show that the lack of strong transversality eventually  still does allow for a rescaling argument.
\smallskip

The article is organized as follows: In Section \ref{auxiliary} we present two  important auxiliary results: On the one hand our "sublevel lemma"  (Corollary \ref{sublevelset}) that will allow us to control certain sub-level sets, on the other hand Lemma \ref{derivinterpolate} which will allow  to control certain derivatives in the rescaling argument.

In Section \ref{geometry} we relate the  hyperbolic geometry of our surfaces  to the  strong transversality property  which is needed to establish the required bilinear estimates. In particular, we shall derive a ''hyperbolic factorization'' of the crucial transversality  function $\Gamma_{z}(z_1,z_2)$   in \eqref{factori}, which will be of central importance to our approach and will  lead to our notation of ``strong separation'' of caps. Our definition of "exceptional" sets will be based on this notion.  In the final Subsections \ref{indu}  and \ref{movestrip}  of Section \ref{geometry} we shall show how to move such exceptional rectangles  into vertical position and  prepare for the subsequent  rescaling argument.

Motivated by these "exceptional" sets, we will devise a notion of $\alpha$-broadness adapted to our class of surfaces in Section \ref{broad}, and prove our crucial  "Geometric Lemma", and we shall show that any two small caps that are not strongly separated  have to be contained in the part of the surface lying over some rectangle  (of possibly  quite arbitrary direction)  or be contained in a larger cap, thus ensuring that a family of pairwise not strongly separated  caps  is in some sense sparse.

In Section  \ref{reductiontobroad} we reduce our main result to estimates for the broad part of the extension operator, and in Section \ref {proofthm} we outline  the actual polynomial partitioning argument and indicate which changes will be required  compared to previous work, in particular to \cite{Gu16} and \cite{bmvp20b}. Here, in particular, we will be brief and only highlight the steps that differ from previous work.

\bigskip

\noindent\textsc{Convention:}
Unless  stated otherwise, $C > 0$ will stand for an absolute constant whose value may vary
from occurrence to occurrence. We will use the notation $A\sim_C B$  to express  that
$\frac{1}{C}A\leq B \leq C A$.  In some contexts where the size of $C$ is irrelevant we
shall drop the index $C$ and simply write $A\sim B.$ Similarly, $A\lesssim B$ will express
the fact that there is a constant $C$ (which does not depend  on the relevant quantities
in
the estimate) such that $A\le C B,$  and we write  $A\ll B,$ if the constant $C$ is
sufficiently small.
\medskip



\setcounter{equation}{0}

\section{Auxiliary results}\label{auxiliary}
In this section, we prove two auxiliary results which will be crucial for our analysis, but which may also be of independent interest.   We begin by recalling the following  theorem from our previous paper \cite{bmvp20a}:

\begin{thm}\label{levelset}
Let $I=[a,b]$ be a compact interval and $g\in C^r(I, \R)$, $r\geq1,$ and put  $C_r:=\|g^{(r)}\|_{L^\infty(I)}.$ Then there exists a decomposition of
$\{g\neq0\}$ into pairwise disjoint intervals $J_{\lambda,\iota}$,  where $\lambda$ ranges over the set of all positive dyadic numbers $\lambda\leq \|g\|_\infty,$  and where for any given $\lambda,$ the index $\iota$ is  from some index set $\mathcal J_\lambda$, such that the following hold true:
\begin{enumerate}
\item $|\J_\lambda|\leq 10r\big(1+|I|C_r^{1/r} \lambda^{-1/r}\big) \lesssim 1+\lambda^{-1/r}$.
\item For any $\lambda$ dyadic, $\iota\in\J_\lambda$ and any $t\in J_{\lambda,\iota}$ we have $\frac12\lambda <|g(t)|<4\lambda.$
\end{enumerate}
\end{thm}

The following corollary will  later allow us  to control certain sub-level sets  which will be important  for our  definition of broad points.

\begin{cor}\label{sublevelset}
Let $I=[a,b]$ be a compact interval and $g\in C^r(I, \R)$, $r\geq1,$ and put  $C_r:=\|g^{(r)}\|_{L^\infty(I)}.$ Then for any $0< \lambda\leq \|g\|_\infty,$  there is a finite family of pairwise disjoint intervals $I_{\lambda,i}$,  where  the index $i$ is  from some  index set $\mathcal I_\lambda$, so that the following hold true:
\begin{enumerate}
\item $|\I_\lambda|\leq 30 r\big(1+|I|C_r^{1/r} \lambda^{-1/r}\big) \lesssim 1+\lambda^{-1/r}$.
\item If we denote by $V_\la$ the union $\bigcup\limits_{i\in \mathcal I_\la}I_{\la,i}$ of all the intervals $ I_{\la,i},$ then
 $$
 \{|g|< \la\}\subset V_\la\subset \{|g|<8\la\}.
 $$
\end{enumerate}
\end{cor}

\begin{proof}
It is sufficient to  prove this for any  dyadic number $\la,$ however, with the slightly  stronger estimates
\begin{equation}\label{levelest1}
 \{|g|< \la\}\subset V_\la\subset \{|g|<4\la\}.
\end{equation}
We shall consider $I$ to be endowed with its relative topology from $\R.$
Consider then the open subset $U_\la:=\{|g|<\la\}$  of $I.$ We decompose it into its connected components $U_{\la,\nu},$ where $\nu$ will be from an at most countable index set. The $U_{\la,\nu}$ are then open subintervals of $I.$ Observe that if such an interval $U_{\la,\nu}$ has endpoints $\alpha<\beta,$ then $|g(\alpha)|=\la$ if $\alpha>a,$ and  $|g(\beta)|=\la$ if $\beta<b.$ Moreover, if $\alpha=a$ and $\beta=b,$  the case where
 $|g(\alpha)|<\la$  and $|g(\beta)|<\la$ cannot arise, since then we would have $U_{\la,\nu}=I$ and  $\|g\|_\infty<\la,$ contradicting our assumptions.
Thus we see that every $U_{\la,\nu}$ will have at least one endpoint, say $t_{\la,\nu},$ such that $|g(t_{\la,\nu})|=\la.$
\smallskip

Next, according to Theorem \ref{levelset}, the point $t_{\la,\nu}$ must be contained in either  one of the intervals $J_{\la,\iota},$ or in  one of the intervals $J_{\la/2,\iota}.$ This interval is unique, and we denote it by $J_{\la}^\nu.$ Observe that then also $U_{\la,\nu}\cup J_\la^\nu$ is an interval, and that $|g| <4\la$ on $U_{\la,\nu}\cup J_\la^\nu.$

Let us finally put
$$
V_\la:=\bigcup\limits_{\nu}U_{\la,\nu}\cup J_\la^\nu.
$$
Then clearly $\{|g|< \la\}=U_\la\subset V_\la,$ and $|g|< 4\la$ on $V_\la.$ Moreover, if we decompose $V_\la=\bigcup\limits_{i\in \mathcal I_\la}I_{\la,i}$ into its connected components $I_{\la,i},$ then each $I _{\la,i}$ is an interval. And, if $t$ is any point in $I _{\la,i},$ then there is some $\nu_i$ so that
$t\in U_{\la,\nu_i}\cup J_\la^{\nu_i},$ so that   $U_{\la,\nu_i}\cup J_\la^{\nu_i}\subset I _{\la,i}.$ The mapping $I _{\la,i}\mapsto J_\la^{\nu_i}$ is clearly injective, since the intervals $I _{\la,i}$ are pairwise disjoint, so that $|\I_\lambda|\leq |\J_{\la/2}|+|\J_\lambda|.$ The estimate in (i) follows thus from estimate (i) in Theorem \ref{levelset}.
\end{proof}

 The following remark shows that we may even assume that the intervals from  Corollary \ref{sublevelset} are not too short. We shall not directly make use of this remark, but the same idea will be used later in Section \ref{broad} to show that we may choose the rectangles in our definition of $\alpha$-broadness sufficiently long.
\begin{remark}\label{thicklevelest}
If $\|g'\|_\infty\le 1,$ we may even assume that all the intervals $I_{\la,i}$ in Corollary \ref{sublevelset} have length $|I_{\la,i}|\ge C\la.$

More precisely, assume that $g\in C^r(I, \R)$ satisfies the assumptions of Corollary  \ref{sublevelset}, that $C>0$ is given, and that in addition
$\|g'\|_\infty\le 1.$ Then for any $0< \lambda\leq \|g\|_\infty$ such that $C\la\le |I|,$  there is a finite family of pairwise disjoint open intervals $I_{\lambda,i}$ of length   $|I_{\la,i}|\ge C\la,$ where  the index $i$ is  from some  index set $\mathcal I_\lambda$, so that the following hold true:
\begin{enumerate}
\item $|\I_\lambda|\leq 30 r\big(1+|I|C_r^{1/r} \lambda^{-1/r}\big) \lesssim 1+\lambda^{-1/r}$.
\item If we denote by $V_\la$ the union $\bigcup\limits_{i\in \mathcal I_\la}I_{\la,i}$ of all the intervals $ I_{\la,i},$ then
 $$
 \{|g|< \la\}\subset V_\la\subset \{|g|<(8+C)\la\}.
 $$
\end{enumerate}
\end{remark}
\begin{proof}
Let us denote for $\de>0$ by $A^\de:=(A+(-\de,\de))\cap I$ the $\de$-thickening  within the interval $I$ of any subset $A$ of $I.$ For this proof, let us endow the quantities appearing in Corollary \ref{sublevelset} with a superscript $\tilde{},$ so that for instance $\tilde I_{\la,\tilde i}, \tilde i\in \tilde{ \mathcal I_\lambda}$ denote the intervals devised in this corollary. Then clearly, with $\de:=C\la,$
 $$
 \{|g|< \la\}\subset  \{|g|< \la\}^ \de\subset V_\la^\de=\bigcup\limits_{\tilde i\in \tilde{\mathcal I_\la}}(\tilde I_{\la,\tilde i})^\de\subset \{|g|<8\la\}^\de
 \subset \{|g|<(8+C)\la\}.
 $$
 Let us again decompose the set $V_\la^\de=\bigcup\limits_{i\in \mathcal I_\la}I_{\la,i}$ into its connected components $I_{\la,i}.$  Since $V_\la^\de$ is open, the $I_{\la,i}$ are  open intervals. Moreover, clearly any $I_{\la,i}$ must contain at least one of the intervals $(\tilde I_{\la,\tilde i})^\de,$ so that its length must at least be $\de=C \la,$ and since the mapping from $i$ to the chosen index $\tilde i$ is injective, we see that
 $|\I_\lambda| \le |\tilde{\I_\lambda}|.$
\end{proof}

\medskip
The next lemma will become important for certain induction on scales arguments.
\begin{lemma}\label{derivinterpolate} Let $g\in C^k(I,\R),$ where $I$ is an interval of length $|I|=b,$ and $k\in\N,$ and assume that
\begin{equation}\label{apriori}
\|g^{(m)}\|_\infty\le c_m,\quad  m=0, \dots, k, \quad \text{ where} \quad\ c_0<1.
\end{equation}

Let $\ve >0,$ and assume that $k\ge 1/\ve,$ and that  $b\le (c_0)^\ve.$ Then there are constants $\tilde C_m(\ve)\ge 0$ depending only on $\ve $ and the constants $c_1, \dots, c_k$ (increasing with the values of the $c_j$), but not on $b$ and $c_0,$ so that
\begin{equation}\label{aposteriori}
\|g^{(m)}\|_\infty\le c_0\, \tilde C_m(\ve) b^{-m}, \qquad  m=0, \dots, k.
\end{equation}
\end{lemma}

\begin{remark} In our later application, we shall have  $c_1,\dots ,c_k\sim 1$ and $c_0\ll 1$, so that  for $m\ge 1$ and $c_0$ sufficiently small  the estimates in \eqref{aposteriori} are stronger than the a priori estimates from \eqref{apriori}, at least when $m$ is not too large.
\end{remark} 
\begin{proof}
We scale by setting $h(x):= g(bx).$ Then, after translation, we may assume that $I=[0,1],$ and that
\begin{equation}\label{deriv1}
\|h^{(j)}\|_\infty\le c_jb^j,\quad  j=0, \dots, k,
\end{equation}
and what we need to show is that there are constants $\tilde C_m(\ve)\ge 0$ as above such that
\begin{equation}\label{deriv2}
\|h^{(m)}\|_\infty\le  c_0\, \tilde C_m(\ve), \qquad m=0, \dots, k.
\end{equation}
To this end, choose $M=M(\ve)\in\N$ minimal so that $M\ge 1/\ve.$ Then $M\le k.$ Assume $m\le k.$

a) If $m\ge M,$ then
$$
c_m b^m\le c_mb^M\le c_mc_0^{M\ve}\le c_m c_0,
$$
so we may choose $\tilde C_m(\ve):= C_m.$ Notice that in particular $\tilde C_M(\ve):=c_{M(\ve)}.$

\medskip
b) Assume next that $0\le m\le M=M(\ve).$  We know already that
\begin{eqnarray}\label{deriv3}
\|h\|_\infty&\le&  c_0,\\
\|h^{(M)}\|_\infty&\le& c_0\, \tilde C_M(\ve).\label{deriv4}
\end{eqnarray}

The estimates in \eqref{deriv2} for $0< m<M$  then follow from these two estimates by interpolation. Let us give an elementary argument for this claim.
Fix $m$ with $0\le m\le M-1.$

We first claim that  \eqref{deriv3} implies that for any $j=0,\dots,m$ there are points
$$
t^j_0<t^j_1<\cdots <t^j_{2^{m-j}-1}
$$
in $[0,1]$ such that $t^j_{i+1}-t^j_{i}\ge 2^{-m} $ and
\begin{equation}\label{deriv5}
|h^{(j)}(t^j_i)|\le c_0 2^{(m+1)j}
\end{equation}
for every $i.$
This is easily proved by induction on $j.$

If $j=0,$ we may choose $t^0_i:= i 2^{-m}, \ i=0,\dots, 2^m-1.$  And, assuming that the claim  holds for $j,$ by the induction hypothesis  and the mean value theorem we can find points $t^{j+1}_i\in (t^j_{2i}, t^j_{2i+1})$ so that
$$
|h^{(j+1)}(t^{j+1}_i)|=\Big|\frac{h^{(j)}(t^{j}_{2i+1})-h^{(j)}(t^{j}_{2i})}{t^j_{2i+1}-t^j_{2i}}\Big|\le \frac {2\cdot c_02^{(m+1)j}}{2^{-m}}=c_02^{(m+1)(j+1)}.
$$
Moreover, since
$$
t^{j+1}_i<t^j_{2i+1}<t^j_{2i+2}<t^{j+1}_{i+1},
$$
where  $t^j_{2i+2}-t^{j}_{2i+1}\ge 2^{-m},$ we also have that  $t^{j+1}_{i+1}-t^{j+1}_{i}\ge 2^{-m}.$

\medskip
In particular, for $j=m,$ we find a $t^m:=t^m_0$ so that $|h^{(m)}(t^m)|\le c_0 2^{(m+1)m}.$  Then, for any $t\in [0,1],$ we have
$$
|h^{(m)}(t)|\le |h^{(m)}(t^m)|+|h^{(m)}(t)-h^{(m)}(t^m)|\le c_0 2^{(m+1)m}+\|h^{(m+1)}\|_\infty\cdot 1,
$$
so that
\begin{equation}\label{deriv6}
\|h^{(m)}\|_\infty\le  c_0 2^{(m+1)m}+\|h^{(m+1)}\|_\infty.
\end{equation}
Now we can use ``downward induction'' on $m,$ starting with $m=M-1,$  to prove \eqref{deriv2}. Indeed, when $m=M-1,$ then by \eqref{deriv6} we have
$$
\|h^{(M-1)}\|_\infty\le  c_0 2^{M(M-1)}+\|h^{(M)}\|_\infty\le c_0 (2^{M(M-1)} +\tilde C_M(\ve))=: c_0 \tilde C_{M-1}(\ve)).
$$
Finally, we can  pass from $m$ to $m-1$  by means of our induction hypothesis on $m$ and \eqref{deriv6}:
$$
\|h^{(m-1)}\|_\infty\le   c_0 2^{m(m-1)}+\|h^{(m)}\|_\infty\le c_0 2^{m(m-1)}+c_0 \tilde C_{m}(\ve)=: c_0\tilde C_{m-1}(\ve).
$$
\end{proof}


\setcounter{equation}{0}
\section{Geometric background on strong transversality}\label{geometry}
Assume that $\phi\in \Hyp^M(\Sigma), M\ge 3,$ and recall that $S$ is the graph of $\phi.$

\subsection{Strong transversality for bilinear estimates}

We  begin by  recalling some   facts about what kind of ``strong transversality'' is required for establishing suitable bilinear estimates.

Following \cite{lee05}, given two open subsets $U_1,U_2\subset \Sigma,$ we consider the quantity
\begin{equation}\label{transs}
\Gamma_{z}(z_1,z_2,z_1',z_2'):=	\left\langle
(D^2\phi(z))^{-1}(\nabla\phi(z_2)-\nabla\phi(z_1)),\nabla\phi(z_2')-\nabla\phi(z_1')\right\rangle
\end{equation}
for  $z_i=(x_i,y_i),\, z'_i=(x'_i,y'_i)\in U_i\, , i=1,2$, and $z=(x,y)\in U_1\cup
U_2.$ Bilinear estimates have constants depending  only
on upper bounds for the
derivatives of $\phi$ and on lower bounds of (the modulus of) \eqref{transs}.    As in \cite{bmvp20b}, for our estimates it will be enough to have lower bounds only for $z\in U_2$ (or only for $z\in U_1$). If $U_1$ and
$U_2$ are sufficiently small (with sizes depending on upper bounds of the first and second
order derivatives of $\phi$ and a lower bound for the Hessian determinant of $\phi$) this
condition  reduces to the estimate
\begin{equation}\label{Gammalow}
|\Gamma_{z}(z_1,z_2)|\geq c>0,
\end{equation}
for $z_i=(x_i,y_i)\in U_i$, $i=1,2$, $z=(x,y)\in U_2$, where
\begin{equation}\label{trans}
\Gamma_{z}(z_1,z_2):=	\left\langle
(D^2\phi(z))^{-1}(\nabla\phi(z_2)-\nabla\phi(z_1)),\nabla \phi(z_2)-\nabla\phi(z_1)\right\rangle.
\end{equation}
In contrast to \cite{bmv20}, \cite{bmvp19}, \cite{bmvp20a}, where we had to devise quite specific  ``admissible pairs'' of sets $U_1, U_2$ for our bilinear estimates, as in \cite{bmvp20b} we shall here only have to consider ``caps'' (cf. Subsection \ref{caps}) $\tau_1, \tau_2$ for $U_1,U_2,$ and the required bilinear estimates will be a of somewhat different nature. Nevertheless,  the  geometric transversality conditions that we need here will be the same.

We shall next exploit the hyperbolicity assumption on $S$ in order to gain a better understanding of $\Gamma_{z}(z_1,z_2,z_1',z_2')$ for such surfaces. In particular, we shall derive the ``hyperbolic factorization'' \eqref{factori} which will be of central importance.

\subsubsection{Null vectors for $D^2\phi$}\label{nullvectors}

Recall  that $\phi\in \Hyp^M(\Sigma), M\ge 3.$ We put $H:=\phi_{xy}^2-\phi_{xx}\phi_{yy},$ so that $-H$ is the Hessian determinant of $\phi.$
From \eqref{phiest} we easily deduce that $|H(z)-1|\le 10^{-4}$ for every $z\in\Sigma.$
\smallskip

It is easy to check that we  then explicitly have
\begin{eqnarray} \nonumber
-H(z)\Gamma_{z}(z_1,z_2)&=&\phi_{yy}(z)\big(\phi_x(z_2)-\phi_x(z_1)\big)^2 +\phi_{xx}(z)\big(\phi_y(z_2)-\phi_y(z_1)\big)^2 \\
& -&2\phi_{xy}(z)\big(\phi_x(z_2)-\phi_x(z_1)\big)\big(\phi_y(z_2)-\phi_y(z_1)\big).
	  \label{gammaz}
\end{eqnarray}

Let us further introduce the functions on $\Sigma$ defined by
\begin{eqnarray*}
A(z):=\frac {\phi_{yy}}{\phi_{xy}+\sqrt{H}}(z), \quad B(z):=\frac {\phi_{xx}}{\phi_{xy}+\sqrt{H}}(z).
\end{eqnarray*}

Note that
\begin{eqnarray}\label{onAB}
 \begin{split}
1+AB&=2\frac{\phi_{xy}}{\phi_{xy}+\sqrt{H}}, \ \ 1-AB=2\frac{\sqrt{H}}{\phi_{xy}+\sqrt{H}}, \\
 \frac {A}{1+AB}&=\frac{\phi_{yy}}{2\phi_{xy}},\  \quad \frac {B}{1+AB}=\frac{\phi_{xx}}{2\phi_{xy}},
\end{split}
\end{eqnarray}
and that \eqref{phiest} implies that
\begin{equation}\label{phiest2}
|\phi_{xy}+\sqrt{H}-2|\le 10^{-4}, \ |\phi_{xx}|, |\phi_{yy}|\le 10^{-4}\quad  \text{on} \ \Sigma,
\end{equation}
so that $|A(z)|, |B(z)|\le 10^{-3}.$
\medskip

$A$ and $B$ are in fact closely linked with the geometry of the surface $S.$ Indeed, consider the vectors $\om:= (-A(z),1)$ and $\nu:=(1,-B(z)).$ Then one checks easily that these two vectors form a {\it basis of  null vectors} of the Hessian matrix $D^2\phi(z),$  i.e., for every  $z\in\Sigma,$ we have

\begin{equation}\label{nullvector}
(-A(z),1) D^2\phi(z) \trans (-A(z),1)=0 \quad \text{and}\quad (1, -B(z)) D^2\phi(z) \trans (1,-B(z))=0.
\end{equation}

For fixed $z,$ let us therefore set
\begin{equation}\label{Tz}
T:= T_{z}:=\left(
\begin{array}{cc}
1  &   -A(z)   \\
  -B(z)   &  1
\end{array}
\right).
\end{equation}
Then clearly
\begin{eqnarray*}
(\xi_1,\xi_2) \trans T D^2\phi(z)\, T \, \trans (\eta_1,\eta_2)&=& (\xi_1\nu+\xi_2\om)D^2\phi(z) \trans (\eta_1\nu +\eta_2\om)\\
&=&q(z) (\xi_1\eta_2+\xi_2\eta_1),
\end{eqnarray*}
where $q(z):=\om D^2\phi(z) \trans \nu.$ This shows that
\begin{equation}\label{normalformbyT}
\trans TD^2\phi(z) T=q(z) \left(
\begin{array}{cc}
 0 &    1  \\
 1   &0
\end{array}
\right).
\end{equation}

Moreover, we have
\begin{eqnarray*}
q(z)=(-A(z),1) D^2\phi(z)\trans (1,-B(z))=\big(-A\phi_{xx}+(1+AB) \phi_{xy} -B\phi_{yy}\big)(z).
\end{eqnarray*}
And, by   our definitions of $A$ and $B,$   and \eqref{onAB}, we  easily see that
\begin{eqnarray*}
-A\phi_{xx}+(1+AB) \phi_{xy} -B\phi_{yy}=2 \frac {H}{\phi_{xy}+\sqrt{H}},
\end{eqnarray*}
so that
\begin{equation}\label{q}
q(z)=2 \frac {H}{\phi_{xy}+\sqrt{H}}(z).
\end{equation}
This implies in particular that  $|q(z)-1|\le 10^{-3}.$ Note also that by \eqref{onAB}\begin{equation}\label{jacobian}
\det T_{z}=1-A(z)B(z)= 2 \frac {\sqrt{H}}{\phi_{xy}+\sqrt{H}}(z)= \frac{q(z)}{\sqrt{H(z)}}\sim q(z)\sim 1.
\end{equation}

\subsubsection{Back to $\Gamma_{z}(z_1,z_2,z_1',z_2')$}\label{Gamma}

Observe next that \eqref{normalformbyT} implies that
$$
(D^2\phi(z))^{-1} =\frac 1{q(z)}T \left(
\begin{array}{cc}
 0 &    1  \\
 1   &0
\end{array}
\right)\trans T.
$$
Thus,
\begin{eqnarray}\nonumber
&&(\xi_1,\xi_2) ( D^2\phi(z))^{-1} \trans (\eta_1,\eta_2)\\
&&=\frac 1{q(z)}\Big[ (\xi_1-B(z) \xi_2)(\eta_2-A(z) \eta_1)+(\eta_1-B(z) \eta_2)(\xi_2-A(z) \xi_1)\Big].\label{factorinversHessian}
\end{eqnarray}
If we  accordingly define  the functions
\begin{eqnarray*}
t^1_{z}(z_1,z_2)&:=&\phi_x(z_2)-\phi_x(z_1)-B(z)(\phi_y(z_2)-\phi_y(z_1)),\\
t^2_{z}(z_1,z_2)&:=&\phi_y(z_2)-\phi_y(z_1)-A(z)(\phi_x(z_2)-\phi_x(z_1)),
\end{eqnarray*}
then the identity \eqref{factorinversHessian} shows that we may re-write
\begin{equation}\label{Gammanew}
\Gamma_{z}(z_1,z_2,z_1',z_2')=\frac 1{q(z)}\Big[ t^1_{z}(z_1,z_2)\cdot t^2_{z}(z'_1,z'_2)+t^1_{z}(z'_1,z'_2)\cdot t^2_{z}(z_1,z_2)\Big].
\end{equation}

In particular,  we obtain the following ``hyperbolic factorization'':
\begin{equation}\label{factori}
\Gamma_{z}(z_1,z_2)=\tfrac 2{q(z)}\cdot t^1_{z}(z_1,z_2)\cdot t^2_{z}(z_1,z_2),
\end{equation}
where the first factor $2/q(z)$ is of size 2, more precisely  $|2/q(z)-2|\le 10^{-3},$ so that it is irrelevant.

Note  also that, e.g.,
\begin{equation}\label{tdiff}
t^2_{z_1}(z_1,z_2)-t^2_{z_2}(z_1,z_2)=(A(z_1)-A(z_2))(\phi_x(z_2)-\phi_x(z_1)),
\end{equation}
and  that
 \begin{equation}\label{symtrans}
 	{t}^i_{z}(z_1,z_2)=-{ t}^i_{z}(z_2,z_1), \quad i=1,2.
\end{equation}

\medskip

\subsubsection{Caps and the basic decomposition of $S$}\label{caps}
\begin{defn}\label{strip-broad} 
Fix $K\gg 1$ to be a large dyadic number, and $\mu\ge 1$ real (the reasons for this notation will be clarified in Section \ref{proofthm}). 

Given  $K$ and $\mu,$ following   \cite{Gu16} we  shall consider a given covering of $\Sigma=[-1,1]\times[-1,1]$ by $K^2$ disjoint squares (called {\it caps}) $\tau$ of side length $\mu^{1/2}K^{-1},$ whose  centers are $K^{-1}$ separated. It can then happen that such a cap $\tau$  is no longer  contained in $\Sigma;$  in that case, we  truncate it by replacing it with its  intersection with $\Sigma.$
  Note that one usually envisions  caps to be subsets of the hypersurface $S;$ for our purposes, however,  it is more convenient  to work with caps $\tau\subset\Sigma,$ which then can be identified with the corresponding caps $\{(z,\phi(z)): z\in \tau\}$ on $S.$
 
Observe that for  $\mu=1,$ this includes in particular  the case of the  covering of $\Sigma$ by caps  $\tau$ which are pairwise disjoint in measure - this is what we had called the   {\it basic decomposition of $S$ into caps} in \cite{bmvp20b}. If $f$ is a given function on $\Sigma,$ we had then defined $ f_\tau:=f\chi_\tau.$  

For general $\mu\ge 1,$ motivated by Guth's inductive argument in  \cite{Gu16}, we shall, however, only assume that $f_\tau$ is a function such that $\supp f_\tau\subset \tau.$  Actually, Guth assumes more generally that $\tau$ is a cap of side length $r_\tau$ with $K^{-1}\le r_\tau\le \mu^{1/2} K^{-1},$ but since we are only assuming that $f_\tau$ is supported in $\tau,$ we can then as well replace $\tau$  by a larger cap of side length $\mu^{1/2}K^{-1},$ as we did.

\end{defn}

Assume now that $\tau_1\ne \tau_2$ are two different caps, with centers $z^c_1=(x^c_1,y^c_1),$  respectively  $z^c_2=(x^c_2,y^c_2),$  of side length $\mu^{1/2} K^{-1}.$ 
The previous  definition of  strong transversality  motivates the following
\begin{defn}
{\rm
We say that $\tau_1\ne \tau_2$ are {\it strongly separated} if
$$
\max\{\min\{|t^1_{z^c_1}(z^c_1,z^c_2)|,
|t^2_{z^c_1}(z^c_1,z^c_2)|\},\min\{|t^1_{z^c_2}(z^c_1,z^c_2)|,|t^2_{z^c_2}(z^c_1,z^c_2)|\}\}\ge 50\mu^{1/2}K^{-1}.
$$}
\end{defn}

We shall distinguish between the cases where $|y^c_2-y^c_1|\ge|x^c_2-x^c_1|,$ and where $|x^c_2-x^c_1|\ge|y^c_2-y^c_1|.$ Let us mostly concentrate on the first case; the other case can be treated in the same way be interchanging the roles of $x$ and $y.$  So,  for the rest of this section, let us make the following
\medskip

\noindent\bf Assumption 1. \label{assumption1}\rm Assume  that $|y^c_2-y^c_1|\ge|x^c_2-x^c_1|.$

\smallskip
\begin{remark}\label{Gammasize}
If the caps $\tau_1$ and $\tau_2$ are strongly separated, so that, say, $|t^1_{z^c_2}(z^c_1,z^c_2)|\ge 50\mu^{1/2}K^{-1}$  and
$|t^2_{z^c_2}(z^c_1,z^c_2)|\ge 50\mu^{1/2}K^{-1},$ then
\begin{equation}\label{Gammavar}
|\Gamma_{z}(z_1,z_2,z_1',z_2')|\ge 4 \mu K^{-2} \quad \text{for all} \quad z_1,z'_1\in \tau_1,  \, z, z_2,z'_2\in \tau_2.
\end{equation}
\end{remark}
This result generalizes  the corresponding result in Remark 4.8 of \cite{bmvp20b}, whose proof easily extends to our present situation by means of the identity \eqref{Gammanew}. It will allow us  to establish favorable bilinear estimates later on for the contributions by the tangential terms associated to the cells arising  in Guth's cell decomposition.

\subsection{Not strongly separated caps}\label{notstrong}

Assume now that $\tau_1$ and $\tau_2$ are not strongly separated.

\medskip
{\bf  Case A.} Assume that $|y^c_2-y^c_1|\le  100 \mu^{1/2}K^{-1}.$
Then, by Assumption 1, also  $|x^c_2-x^c_1|\le  100\mu^{1/2}K^{-1}.$  Both caps  are then contained in a cap of slightly bigger size  
$100\mu^{1/2}K^{-1}\le 100\mu^{1/2}K^{-1/4}.$ \smallskip
\medskip

Let us therefore assume from here on that  $|y^c_2-y^c_1|>  100\mu^{1/2}K^{-1}.$

Observe that our assumptions on  $\phi$  in combination with Assumption 1 then easily imply that $|t^1_{z}(z_1,z_2)|\sim |y_2-y_1|$ for every $z_1\in\tau_1, z_2\in\tau_2$ and $z\in \tau_1\cup \tau_2,$  and thus we may assume that
$$
\min\{|y^c_2-y^c_1|,|t^2_{z^c_1}(z^c_1,z^c_2)|\}\le  100\mu^{1/2}K^{-1}\text{ and } \min\{|y^c_2-y^c_2|,|t^2_{z^c_2}(z^c_1,z^c_2)|\}\le  100\mu^{1/2}K^{-1}.
$$
In particular, we have 
\begin{equation}\label{tbounds}
|t^2_{z^c_1}(z^c_1,z^c_2)| \le  100\mu^{1/2}K^{-1} \text{ and } |t^2_{z^c_2}(z^c_1,z^c_2)|\le  100\mu^{1/2}K^{-1}.
\end{equation}

Note also that by \eqref{tdiff} 
\begin{equation}\label{factorcenter}
|t^2_{z^c_1}(z^c_1,z^c_2)-t^2_{z^c_2}(z^c_1,z^c_2)|\sim|A(z^c_1)-A(z^c_2)||y^c_2-y^c_1|,
\end{equation}
with constants very close to $1.$

We shall therefore distinguish two further cases.

\smallskip

{\bf  Case B.} Assume that $|y^c_2-y^c_1|>  100 \mu^{1/2}K^{-1}$ and  $|A(z^c_1)-A(z^c_2)|> \mu^{1/2}K^{-3/4}.$

Then, by  \eqref{tbounds} and \eqref{factorcenter},  $|y^c_2-y^c_1|\le 300 K^{-1/4}\le 300 \mu^{1/2}K^{-1/4},$ and arguing as before we see that both caps are contained in a cap of size   $400 \mu^{1/2}K^{-1/4}.$
\smallskip

 This leaves us with the case where $|y^c_2-y^c_1|>  100 \mu^{1/2}K^{-1}$ and  $|A(z^c_1)-A(z^c_2)|\le  \mu^{1/2}K^{-3/4}.$  Actually, in what follows, we shall  not really make use of the condition $|y^c_2-y^c_1|>  100 \mu^{1/2}K^{-1}$  and shall therefore henceforth concentrate on 
\smallskip

{\bf  Case C.} Assume that 
\begin{equation}\label{nonsepcase}
|t^2_{z^c_1}(z^c_1,z^c_2)|\}\le  100\mu^{1/2}K^{-1} \quad \text{and} \quad |A(z^c_1)-A(z^c_2)|\le  \mu^{1/2}K^{-3/4},\ .
\end{equation}

\smallskip

\noindent\bf Notation. \rm We fix a point $z_1$ (which would be the point $z^c_1$ in Case C), and set $A_1:=A(z_1).$ Then, we define
\begin{eqnarray}\label{curves}
R_I&:=&\{z\in\Sigma: |t_{z_1}^2 (z_1,z)|\le 100 \mu^{1/2}K^{-1}\},\\
 R_{II}&:=&\{ z\in\Sigma: |A(z)-A_1|\le \mu^{1/2}K^{-3/4}\}.
\end{eqnarray}

\subsubsection{Level curves of $t_{z_1}^2 (z_1,\cdot)$} Let us fix $z_1\in \Sigma,$ and let us abbreviate ${\bf t}_{z_1}(x,y):= t_{z_1}^2 (z_1,(x,y)).$  From our definition of ${\bf t}_{z_1}$  we compute that
\begin{eqnarray}\label{nablat2}
\nabla {\bf t}_{z_1}(z)&=&(\phi_{xy}-A_1\phi_{xx}, \phi_{yy}-A_1 \phi_{xy})(z)=(1,0)+O(10^{-5}).
\end{eqnarray}

\begin{lemma}\label{curvesI}
Let $\al:=\min\limits_{z\in\Sigma} t_{z_1}^2 (z_1,z), \, \beta:=\max\limits_{z\in\Sigma} t_{z_1}^2 (z_1,z).$ There exists a  $C^M$-function
$h:[\al, \beta]\times[-1,1]\to \R$ such that  the  curves $\ga_{I,v}(y):=(h(v,y),y), y\in [-1,1],$ with $v\in [\al, \beta],$ are level curves of the function
$t_{z_1}^2 (z_1,\cdot)$ which fibre the set $\Sigma_{z_1}:=\{(h(v,y),y):[(v,y)\in[\al, \beta]\times[-1,1]\}$ into (``almost vertical'') curves, and $\Sigma\subset \Sigma_{z_1}
\subset 2\Sigma.$ Moreover, the mapping $\mathcal H:[\al, \beta]\times[-1,1]\to \Sigma_{z_1}, (v,y)\mapsto  (h(v,y),y),$ is a $C^M$-diffeomorphism, and more precisely we have that
  $h_y(y,v)=O(10^{-4}), \  h_v(y,v)=1+O(10^{-4}).$

\end{lemma}

\begin{proof}
Consider the mapping $\mathcal G: (x,y)\mapsto ({\bf t}_{z_1}(x,y), y).$ Then, by \eqref{nablat2}
$$D\mathcal  G(x,y)=\left(
\begin{array}{cc}
 \phi_{xy}-A_1\phi_{xx} &  \phi_{yy}-A_1 \phi_{xy})    \\
 0 &  1
\end{array}
\right)(z)
=\left( \begin{array}{cc}
1& 0   \\
 0 &  1
\end{array}
\right)+O(10^{-5}).
$$
Therefore, the results follows in a straight-forward manner from the inverse function theorem, by patching together local inverse functions. Note that the inverse function $\mathcal  H$ to $\mathcal G$ must be of the form $\mathcal H(v,y)=(h(v,y),y),$ and that $t_{z_1}^2 (z_1,(h(y,v),y))={\bf t}_{z_1}(h(y,v),y)=v.$ This also implies that
\begin{equation}\label{nablah}
0=\pa_x{\bf t}_{z_1}\cdot h_y+\pa_y{\bf t}_{z_1}; \qquad 1=\pa_x{\bf t}_{z_1}\cdot h_v,
\end{equation}
so that, in view of \eqref{nablat2},  $h_y(y,v)=O(10^{-4})$ and $h_v(y,v)=1+O(10^{-4}).$
\end{proof}

Note that this result also implies that horizontal sections of  $R_I$ have length $O(\mu^{1/2}K^{-1}).$
\smallskip

\subsubsection{On the direction of level curves of $t_{z_1}^2 (z_1,\cdot)$ within $R_I\cap R_{II}$ }

The following lemma gives us an important geometric information.
\begin{lemma}\label{anglestable} The set $R_I$ fibers into the level curves $\ga_{I,v}(y):=(h(v,y),y), y\in [-1,1],$ with $|v|\le100 \mu^{1/2}K^{-1}.$  If
$z=(h(v,y),y)$ lies on such a curve, denote by  $X_{z}:=(h_y(v,y),1)$  the corresponding tangent vector at the point $z.$ Then, if $z$ lies also in $R_{II},$  i.e., if $z\in R_I\cap R_{II},$ we have that
$$
|X_z-(-A_1,1)|\le 3\mu^{1/2} K^{-3/4}.
$$
Thus, up to an error of order $O(\mu^{1/2}K^{-3/4}),$ for all points $z$ in $ R_I\cap R_{II}$ the tangent vectors to the level curves of $t_{z_1}^2 (z_1,\cdot)$ point in the same direction given by $(-A_1,1).$
\end{lemma}
\begin{proof}
By Lemma \ref{curvesI} the set $R_I$ fibers into the level curves $\ga_{I,v}(y):=(h(v,y),y), y\in [-1,1],$ with $|v|\le100 \mu^{1/2}K^{-1}.$  Fix any such $v.$ Then $X_{z}=(h_y(v,y),1)$ is a tangent vector of length $1+O(10^{-4})$ to the corresponding curve,   if $z:=(h(v,y),y).$ Note that, by  \eqref{nablah} and \eqref{nablat2},
$$
h_y(v,y)=-\frac{\pa_y{\bf t}_{z_1}}{\pa_x{\bf t}_{z_1}}(z)=-\frac{\phi_{yy}(z)-A_1 \phi_{xy}(z)}{\phi_{xy}(z)-A_1\phi_{xx} (z)}.
$$
Let us compare this  quantity with the one where $A_1$ is replaced by $A(z),$ i.e.,  with
$$
-\frac{\phi_{yy}(z)-A(z) \phi_{xy}(z)}{\phi_{xy}(z)-A(z)\phi_{xx} (z)}.
$$
Our definition of $A(z)$ implies that $\phi_{yy}-A \phi_{xy}=A\sqrt{H}$  and  that
$$
\phi_{xy}(z)-A\phi_{xx} =\frac {H+\phi_{xy} \sqrt{H}}{\phi_{xy}+\sqrt{H}}=\sqrt{H},
$$
so that
\begin{equation}\label{Arelation}
-\frac{\phi_{yy}(z)-A \phi_{xy}(z)}{\phi_{xy}(z)-A\phi_{xx} (z)}=-A(z).
\end{equation}
Therefore, if $z=(h(v,y),y)\in R_{II},$ i.e., if $|A(z)-A_1|\le \mu^{1/2}K^{-3/4}$, by \eqref{normal1}, \eqref{regularity}, we see that $|h_y(v,y)-(-A(z))|\le 2 \mu^{1/2} K^{-3/4},$ hence $|h_y(v,y)-(-A_1)|\le 3 \mu^{1/2}K^{-3/4},$ if $K$ is supposed to be sufficiently large. This implies that $|X_z-(-A_1,1)|\le \mu^{1/2}3K^{-3/4}.$
\end{proof}

\subsection{Moving rectangular boxes into vertical position at the origin}\label{movestrip}
 Suppose that $I$ is a subinterval of $[-1,1]$ of length $b:=|I|$  so that for any $y\in I$ there is some $x_y$  so that the point $z_y:=(x_y,y)$ lies in $R_I\cap R_{II}.$ Lemma \ref{anglestable} then shows that the set $(R_I\cap R_{II})\cap ([-1,1]\times I)$ is  essentially contained in a rectangular box  $L$ of dimension $100 \mu^{1/2}K^{-3/4}\times b,$ pointing in the direction of the vector
$\om:=(-A_1,1).$  Moreover, up to an error of order $O(\mu^{1/2}K^{-3/4}),$  we may replace $A_1$ by  $A(z),$  for any choice of point $z\in L.$

\medskip
Indeed, in Section \ref{broad}, based on Lemma \ref{anglestable} and Corollary  \ref{sublevelset}, we shall devise in a systematic way such kind of rectangular boxes $L,$ whose lengths will in addition satisfy the following condition:
\begin{equation}\label{boxlength}
100 \mu^{1/2}K^{-3/4}\le b\le K^{-\ve'},
\end{equation}
where $\ve'\in(0,\ve)$ will be a fixed, but sufficiently small constant to be  defined later. Moreover,   in view of Remark \ref{rem4.1}, we may also assume 
that $\mu\leq K^{\ve},$  so that $100 \mu^{1/2}K^{-3/4}\leq K^{-\ve} \leq K^{-\ve'}$.

\medskip
Here comes another {\bf crucial observation}: let $z=z_y$ for $y\in I.$ Then, by \eqref{nullvector}, we know that $(-A(z),1) D^2\phi(z) \trans (-A(z),1)=0.$ But, since $z\in R_{II},$ we also have that $|A(z)-A_1|\le \mu^{1/2} K^{-3/4}, $ and thus $\om D^2\phi(z_y) \trans\om= (-A_1,1) D^2\phi(z_y) \trans (-A_1,1)=O(\mu^{1/2}K^{-3/4}).$ And, since the box $L$ is of horizontal width $100 \mu^{1/2}K^{-3/4},$ the same estimate holds throughout $L.$ Consequently, we see that we may assume (with $\om=(-A_1,1)$) that, say,
\begin{equation}\label{nablanulldir}
|\left\langle \om,\nabla   \right\rangle^2 \phi(z)|\le C \mu^{1/2}K^{-3/4} \qquad \text{for all} \ z\in L.
\end{equation}

Moreover, by our assumptions on $\phi,$ clearly we also have
\begin{equation}\label{nablanulldirm}
|\left\langle \om,\nabla  \right\rangle^m \phi(z)|\le c_m  \qquad \text{for all} \ z\in L,
\end{equation}
for all $m=3, \dots,M,$ with constants $c_m\lesssim 10^{-5}.$  Restricting these estimates to lines parallel to $\R \om,$  and applying Lemma \ref{derivinterpolate} to $\left\langle \om,\nabla   \right\rangle^2 \phi$ along these lines, we see that these two estimates imply that
\begin{equation}\label{derivcontrol}
|\left\langle \om,\nabla \right\rangle^m   \phi(z)|\le \mu^{1/2}\tilde C_m(\ve) K^{-3/4} b^{2-m}\qquad \text{for all} \ z\in L, \,m=2, \dots, M.
\end{equation}

Let us now denote by  $z_0=(x_0,y_0)$ the center of our rectangular box $L.$  For simplicity, we may and shall assume  that $A_1=A(z_0).$
\smallskip

Our goal will be to find an affine-linear  transformation
$z=z_0+T\tilde z$ so that for the accordingly transformed function
$\tilde\phi(\tilde z):=\phi(z_0+T\tilde z)$ we have that
\begin{equation}\nonumber
D^2\tilde \phi(0)=q(z_0) \left(
\begin{array}{cc}
 0 &    1  \\
 1   &0
\end{array}
\right),
\end{equation}
where $q(z_0)\in\R.$  Note that in the coordinates $\tilde z,$ the point $z_0$ then corresponds to $\tilde z_0=0.$
\smallskip

To this end, let us put $B_1:=B(z_0),$ and choose for $T$ the matrix $T_{z_0}$ defined by \eqref{Tz}, i.e.,
$$
T:= T_{z_0}:=\left(
\begin{array}{cc}
1  &   -A_1   \\
  -B_1   &  1
\end{array}
\right).
$$
Then, by  \eqref{normalformbyT}, we have indeed that
$$
D^2\tilde\phi(0)=\trans TD^2\phi(z_0) T=q(z_0) \left(
\begin{array}{cc}
 0 &    1  \\
 1   &0
\end{array}
\right),
$$ with $q(z)$ defined as in Subsection \ref{nullvectors}.
Note also that by \eqref{jacobian} the Jacobian determinant of our change of coordinates is given by
$$
\det T_{z_0}=\frac{q(z_0)}{\sqrt{H(z_0)} }\sim q(z_0)\sim1.
$$

In the new affine-linear coordinates $\tilde z,$ denote quantities like $A$, $L,$ etc., by $\tilde A,$ $\tilde L,$ etc.. Then $\tilde A(\tilde z_0)=\tilde A(0)=0,$ so that  $\tilde \om=(0,1).$
This corresponds to the following observation: we have
$\tilde\phi(\tilde x,\tilde y)=\phi(x_0+\tilde x-A_1\tilde y,y_0-B_1\tilde x+\tilde y),$  so that $\frac{\pa}{\pa \tilde y}\tilde \phi(\tilde z)=(\left\langle \om,\nabla  \right\rangle \phi)(z_0+T\tilde z).$ This shows that indeed the
 directional derivative $\left\langle \om,\nabla  \right\rangle$ corresponds to the partial derivative with respect to $\tilde y$ in the coordinates $\tilde z=(\tilde x,\tilde y).$  Thus, by \eqref{derivcontrol}, we have that
\begin{equation}\label{derivcontroltilde}
|\pa_{\tilde y}^m\tilde \phi(\tilde z)|\le \tilde C_m(\ve) \mu^{1/2}K^{-3/4} b^{2-m}\qquad \text{for all} \ \tilde z\in \tilde L, \,m=2, \dots, M,
\end{equation}
if $\tilde L$ corresponds to $L$ in the $\tilde z$-coordinates, i.e., $\tilde L=T^{-1}(-z_0+L).$ Note that $\tilde L$ is essentially again a rectangular box of dimension  $100 \mu^{1/2}K^{-3/4}\times b,$ but centered at the origin and  vertical, so that we may assume that $\tilde L$ is contained in $\tilde {\tilde L}:=[-100 \mu^{1/2}K^{-3/4},100 \mu^{1/2}K^{-3/4}]\times [-b,b].$ We  may and shall assume that \eqref{derivcontroltilde} holds even on $\tilde {\tilde L}.$

From \eqref{regularity}, we also get the following estimates on the box $\tilde {\tilde L}:$

\begin{equation}\label{regularitytilde}
\|\pa_{\tilde x}^\al\pa_{\tilde y}^\beta\tilde \phi\|_\infty \le C_m 10^{-5} \quad \text {for} \quad 3\le \al+\beta\le M,
\end{equation}
for constants $C_m>0$ which may possibly be  much bigger than $1$ if $m$ is very large.

Finally we note  that by replacing $\tilde \phi$ with $q(z_0)^{-1} \tilde \phi,$ and subtracting the first order Taylor polynomial  at the origin from $\tilde \phi,$ we may  even assume that
\begin{equation}\label{normal1tilde}
\tilde\phi(0)=0, \ \nabla\tilde\phi(0)=0, \
D^2\tilde\phi(0)=\left(
\begin{array}{ccc}
 0 &   1  \\
 1 &  0   \\
\end{array}
\right).\\
\end{equation}

We remark that our affine-linear coordinate change for passing to the $\tilde z$-coordinates, as well as  the further adjustments to $\tilde \phi$ that we just have explained, have no essential effect on the associated Fourier extension estimates, except that the operator norms may  be bigger by a factor  $C\sim 1.$

Last, but not least, observe also that we may assume that the same type of estimates \eqref{derivcontrol} and \eqref{regularitytilde} will also hold on the double $2\tilde L$ of $\tilde L,$ with constants $\tilde C_m(\ve)$ respectively $C_m$ that  possibly increase  by yet  another factor $C\sim 1.$

\medskip
In the next subsection we shall show how our previous results allow also the carry out later the induction on scales step.

\subsection{The induction on scales step}\label{indu}

For this subsection, we may and shall assume that  $\mu=1.$ Moreover, to simplify the subsequent discussions, let us drop the factor $100 $ from the horizontal length of our box $\tilde L,$ i.e., let us assume that $\tilde L$   is contained in  $\tilde {\tilde L}$ given by $[-K^{-3/4},K^{-3/4}]\times [-b,b],$ and let us correspondingly drop this factor also from \eqref{boxlength}.
\smallskip

In a second step, let us then  scale the $\tilde z$-coordinates by writing $\tilde x=K^{-3/4} x', \ \tilde y=by',$  $z'=(x',y'),$  and
$$
\phi^s(z'):=\frac{ K^{3/4}}{b}\tilde \phi(K^{-3/4} x', by').
$$
Note that in the new coordinates $z',$  $\tilde {\tilde L}$ corresponds to our standard square $\Sigma.$
Then, on $\Sigma,$ we have
$$
\pa_{x'}^\al\pa_{y'}^\beta \phi^s(x',y')=(K^{-3/4})^{\al-1} b^{\be-1} \pa_{\tilde x}^\al\pa_{\tilde y}^\beta\tilde \phi(K^{-3/4} x', by'),
$$
so that in view of  \eqref{normal1tilde}
\begin{equation}\label{normal1s}
\phi^s(0)=0, \ \nabla\phi^s(0)=0, \
D^2\phi^s(0)=\left(
\begin{array}{ccc}
 0 &   1  \\
 1 &  0   \\
\end{array}
\right).\\
\end{equation}

And, if $\al\ge 1$ and  $\alpha+\beta\geq 3$, by  \eqref{regularitytilde} and \eqref{boxlength}, we have 
\begin{equation}\nonumber
\|\pa_{x'}^\al\pa_{y'}^\beta \phi^s\|_\infty 
\le (K^{-3/4})^{\al-1} b^{\beta-1} C_m 10^{-5}
\le  b^{\al-1} b^{\be-1} C_m 10^{-5}
\le C_m 10^{-5} b \le 10^{-5}.
\end{equation}
For the remaining case $\al=0$, we use the improved estimate \eqref{derivcontrol},  which implies that 

\begin{equation}\nonumber
\|\pa_{y'}^\beta \phi^s\|_\infty 
\le (K^{-3/4})^{-1} b^{\be-1} \tilde C_\beta(\ve) K^{-3/4} b^{2-\beta} 
= \tilde C_\beta(\ve) b \le 10^{-5},
\end{equation}
since $b\le K^{-\ve'}$ and we can choose $K$ sufficiently large.

We thus see that
$$
\|\pa_{x'}^\al\pa_{y'}^\beta \phi^s\|_\infty \le 10^{-5} \quad \text {for} \quad 3\le |\al|+|\beta|\le M,
$$
if we assume $K\gg 1$ to be sufficiently large.

\medskip
Actually, if we denote by $2L$ the doubling of $L$ which keeps the center of $L$ fixed,  we may even assume that the estimates \eqref{nablanulldirm}, \eqref{derivcontrol} hold true on $2L,$ with constants bigger by some fixed factor only, and so the same arguments used before show that we may even assume that $\phi^s$ is defined on $2\Sigma,$ and that the previous estimates hold true even on $2\Sigma.$

\smallskip
Thus we see that the function $\phi^s$ lies again in $\Hyp^M(\Sigma).$

\subsubsection{Final rescaling step  in the induction on scales argument}

Recall that we assume  that $\mu=1.$ Explicitly, our construction of $\phi^s$ shows that
\begin{eqnarray}\nonumber
\phi^s(x',y')&=&\frac 1{ q(z_0)}bK^{-3/4} \Big[ \phi(x_0+K^{-3/4} x'-A_1by', y_0-B_1K^{-3/4}x'+by')\\
&-&\phi_x(z_0)(K^{-3/4}x'-A_1by') -\phi_y(z_0) (-B_1K^{-3/4}x'+by') \Big] +\text{constant}.\label{phisexpli}
\end{eqnarray}
Thus, if $f_L:=f\chi_L,$ and if $f^L$ denotes the corresponding function in the $z'$-coordinates, then changing coordinates we obtain
\begin{eqnarray*}
\ext _\phi f_L(\xi)&=&\int f_L(x,y) e^{-i(\xi_1 x+\xi_2 y+\xi_3\phi(x,y))}  \, dx dy\\
&=&C\,  bK^{-3/4}\int f^L(x',y') e^{-i\Phi(x',y'; \xi)} dx'dy',
\end{eqnarray*}
$C\sim 1,$ and where the phase is given by
\begin{eqnarray*}
\Phi(x',y'; \xi)&:=&\xi_3\Big[ b K^{-3/4}\phi^s(x',y') +\phi_x(z_0)( K^{-3/4} x'-A_1by')+\phi_y(z_0)(-B_1K^{-3/4} x'+by') \Big] \\
&+&\xi_1(x_0+K^{-3/4}x'-A_1 by')+\xi_2(y_0-B_1K^{-3/4}x'+by')+\text{constant} \cdot \xi_3.
\end{eqnarray*}
Thus, up to a fixed linear function in $\xi,$ which is irrelevant, we may assume that
\begin{eqnarray*}
\Phi(x',y'; \xi)&=&\xi_3 b K^{-3/4}\phi^s(x',y') +x'K^{-3/4} \big(\xi_1-B_1\xi_2+(\phi_x(z_0)-B_1\phi_y(z_0)) \xi_3\big)\\
&+&y'b\big( \xi_2-A_1\xi_1+(\phi_y(z_0) -A_1\phi_x(z_0))\xi_3\big).
\end{eqnarray*}
This implies that
\begin{equation}\label{Exts}
|\ext_\phi f_L(\xi)|= b K^{-3/4} |\ext_{\phi^s} f^L(S\xi)|,
\end{equation}
where $S\xi$ is defines by
$$
S\xi:=\Big( K^{-3/4}(\xi_1-B_1\xi_2+(\phi_x(z_0)-B_1\phi_y(z_0)) \xi_3, b(\xi_2-A_1\xi_1+(\phi_y(z_0) -A_1\phi_x(z_0))\xi_3), b K^{-3/4} \xi_3\Big).
$$
We shall be interested in estimating $\|\ext_\phi f_L\|_{L^p(B_R)},$ where $B_R$ denotes the Euclidean ball of radius $R$ centered at the origin.
Note that if $\xi\in B_R,$ then $\xi'=S\xi$ lies in the set $B'_R$ defined by
$$
|\xi'_1|\le 3  K^{-3/4} R, \quad |\xi'_2|\le 2b R, \quad |\xi'_3|\le b K^{-3/4} R.
$$
What is important to us is the estimate for the third component $\xi',$ which is bounded by $R':=K^{-3/4}R\ll R.$    This will allow to go from scale $R'$ to scale $R$ as in \cite{bmvp20b}.

Indeed, observe the following estimates, which follow easily from our definition of $f^L$ and \eqref{Exts}:
\begin{eqnarray}\label{passageL}
\begin{split}
\|f^L\|_2&\le (b K^{-3/4})^{-1/2} \|f_L\|_2, \quad \|f^L\|_\infty\le \|f\|_\infty, \\
 \|\ext_\phi f_L\|_{L^p(B_R)}&\le  (b K^{-3/4})^{1-\frac 2p}\|\ext_\phi^sf^L\|_{L^p(B'_R)}.
 \end{split}
\end{eqnarray}
Now assume by   induction hypothesis that
\begin{equation*}
	\|\ext_{\phi^s} f^L\|_{L^{p}(B_{R'})} \leq C_\epsilon R'^\epsilon \|f^L\|_{L^2(\Sigma )}^{2/q}\,\|f^L\|_{L^\infty(\Sigma)}^{1-2/q},
\end{equation*}
where $\ve>0$ is as in Theorem \ref{mainresult}.
By means of Lemma 5.1 in \cite{bmvp20b} we can replace the ball $B_{R'}$ on the left-hand side by $\R^2\times [-R',R']$ and keep the same estimate, with a possibly slightly larger constant $C'_\epsilon.$ In particular, we see that
\begin{equation*}
	\|\ext_{\phi^s} f^L\|_{L^{p}(B'_{R})} \leq C'_\epsilon R'^\epsilon \|f^L\|_{L^2(\Sigma )}^{2/q}\,\|f^L\|_{L^\infty(\Sigma)}^{1-2/q}.
\end{equation*}
Combining this with \eqref{passageL}, we see that
\begin{eqnarray}\nonumber
 \|\ext_\phi f_L\|_{L^p(B_R)}&\le  (b K^{-3/4})^{1-\frac 2p}C'_\epsilon R'^\epsilon \|f^L\|_{L^2(\Sigma )}^{2/q}\,\|f^L\|_{L^\infty(\Sigma)}^{1-2/q}\\ \nonumber
 &\le C'_\epsilon R'^{\ve}(b K^{-3/4})^{\frac 1{q'}-\frac 2p}\|f_L\|_2^{2/q}\|f\|_\infty^{1-2/q}\\
  &\le C'_\epsilon R^\epsilon (K^{-3/4})^{\frac 1{q'}-\frac 2p+\ve}\|f_L\|_2^{2/q}\|f\|_\infty^{1-2/q}, \label{indest}
\end{eqnarray}
since we assume that $p>2q'.$  It is important that the last estimate does not depend on the length $b$ of $L,$ which may vary with $L.$ 
By means of \eqref{indest}, we can now proceed similarly as in  Section 1 of \cite{bmvp20b} and sum these estimates over all boxes $L$ in an appropriate way - for the details of this, we refer to Section \ref{reductiontobroad}.

%
%
%
%
%

\setcounter{equation}{0}
\section{Broad points}\label{broad}
\subsection{Definition of broad points and the underlying family of rectangles}

For the unperturbed hyperbolic paraboloid, the definition of broadness is based on horizontal and vertical strips, since these are the sets which lack strong transversality. Here, we will devise a family $\bar{\mathcal{L}}$ of rectangles (and their intersections) adapted to our perturbed hyperbolic paraboloid.

Let us assume as in Definition \ref{strip-broad} that $K\gg 1.$  Moreover, in view of Remark \ref{rem4.1}, let us also assume that $1\leq\mu\leq K^{\ve/2},$ where $\ve>0$ is as in Theorem \ref{mainresult}. Let further fix an  according  family of caps $\tau$ of side length $\mu^{1/2}K^{-1}.$ Note that as  in  \cite[Theorem 2.4]{Gu16} and \cite[Theorem 2.1]{bmvp20b}, our main goal later will be to prove Theorem \ref{broadtheorem}, in which we are indeed assuming that the caps $\tau$  form the basic decomposition of $\Sigma,$ with $\mu=1.$ However, for the inductive argument which is used to prove this theorem, we shall be forced to consider also cases where $\mu>1.$

\smallskip
As usual, for $0<\alpha<1$, we define a point $\xi$ to be $\alpha$-broad for $\E f$,
if
$$
 \max_{L\in\bar{\mathcal{L}}} |\E f_L(\xi)|\le\alpha|\E f(\xi)|,
$$
where $f_L:=\sum\limits_{\tau\subset L} f_\tau$. We define $Br_\alpha\E f(\xi)$ to be $|\E f(\xi)|$ if $\xi$ is $\alpha$-broad, and zero
otherwise.
\smallskip

 We explain now how to construct the family $\bar{\mathcal{L}}$.
 \medskip

As we saw in Section \ref{notstrong},  Case C,  the most troublesome sets lacking  strong transversality are the intersection of the sets $R_I$ and $R_{II}$ from \eqref{curves}. We recall the functions
\begin{eqnarray*}
A(z)=\frac {\phi_{yy}}{\phi_{xy}+\sqrt{H}}(z), \quad B(z)=\frac {\phi_{xx}}{\phi_{xy}+\sqrt{H}}(z).
\end{eqnarray*}
Note that for $\phi\in\text{Hyp}^3$, we have 
 \begin{equation}\label{nablaA}
	 |\nabla A(z)|,\ |\nabla B(z)| \leq 1
 \end{equation}
for all $z\in\Sigma$.

We shall focus the discussion on $A;$ there is an analogous construction with $B$.
Let $\{A_k\}_k$ be a equidistant decomposition of the interval $A(\Sigma)$ of distance $C\muk$, that is, $A_{k+1}=A_k+C\muk$, where we will choose the constant $C$ later, and  put
$$R_{II}^k:=\{z\in\Sigma:|A(z)-A_k|< C \muk\}.$$ 

The sets $R_{II}^k$ are not pairwise disjoint, but $R_{II}^k$ and $R_{II}^{k'}$ may only overlap if $|k-k'|\leq 1$.
As we saw in Lemma \ref{anglestable}, on $R_I\cap R_{II}$, the tangents to $R_I$ point essentially in a fixed direction.
This suggests to denote for any given $k$ by $\tilde\omega_k$ the unit vector pointing in the direction of $\omega_k:=(-A_k,1),$  and decompose $\Sigma$  into strips 
$$S_{k,j}:=\big[(j-1)C'\muk,(j+1)C'\muk\big]\tilde\omega_k^\perp+\R\tilde\omega_k$$ of thickness $2C'\muk$ and direction $\omega_k,$  indexed  by suitable integers $j\in\Z.$ 
Here, 
$\tilde \omega_k^\perp$ denotes a unit vector  orthogonal to $\omega_k,$ and 
$C'$ denotes yet another suitable constant. For fixed $k$, $S_{k,j}$ and $S_{k,j'}$ do not overlap unless $|j-j'|\leq 1$. 
Of course they can overlap quite a lot for different values of $k$, but we want to consider only the part of $S_{k,j}$ that intersects $R_{II}^k$.   More precise, we define the following  subset $S_{k,j}^0$ of $S_{k,j}:$
$$
S_{k,j}^0:=\{z\in S_{k,j}: (z+\R \tilde\omega_k^\perp)\cap S_{k,j}\cap R_{II}^k\ne \emptyset\}.
$$
It is clear that the connected components of $S_{k,j}^0$ are all rectangles inside $S_{k,j}$ of full width $2C'\muk,$ but  unknown length. Since for technical reasons that will become clear later, we do not want the rectangles to be too short, we set 
$$S_{k,j}^1:=S_{k,j}^0+[-\muk,\muk]\tilde \omega_k.
$$
Then the connected components of $S_{k,j}^1$ are rectangles inside $S_{k,j}$ of full width and length at least $2\muk.$

On the other hand, we want the lengths of these rectangles not to be too long either, and therefore divide $S_{k,j}^1$ into rectangles $\tilde L_{k,j}^i$ of lengths at most $\frac12 K^{-\ve'},$  but at least $\muk$ ($\ve'\ll\ve$ to be determined later), by artificially chopping any connected component that is too long. 
Finally, since that artificial chopping may split a small cap $\tau$ of size $\muk$ into two, we set
$$L_{k,j}^i:=\tilde L_{k,j}^i+[-\muk,\muk]\tilde \omega_k,
$$
so that for fixed $k$ and $j$, the sets $L_{k,j}^i$ may intersect, but only two can overlap at any given point.
Since $2\muk<\frac12 K^{-\ve'}$ for sufficiently small $\ve$, every $L_{k,j}^i$ has length between $\muk$ and $K^{-\ve'}$.

\smallskip
Let $\mathcal{L}_1:=\{L_{k,j}^i\}_{k,j,i}$ denote the set of all these rectangles  (which we often simply shall call ``strips'').

By construction, $\dist(z,R_{II}^k)\leq 2C'\muk$ for all $z\in S_{k,j}^0$; for $z\in L^i_{k,j}$, we still have
$$\dist(z,R_{II}^k)\leq (2C'+2)\muk,$$
so that 
\begin{equation}
	|A(z)-A_k|\leq C\muk + \|\nabla A\|_\infty (2C'+2)\muk \lesssim \muk 
\end{equation}
for all $z\in L_{k,j}^i$.

We summarize the most important properties of our family of rectangles, which follow immediately from their definition:
\begin{remarks}\label{famprop}
 There exist absolute constants $C_1,N_1\ge 1$ such that the following hold true:
 \begin{enumerate}
\item  For  all $k,j,i$ and all $z\in L^i_{k,j},$ 
$$|A(z)-A_k|\leq C_1\muk.$$
\item  At any point, at most $N_1$ sets from $\mathcal{L}_1$ overlap.
\item Every $L\in\mathcal{L}_1$ has length between $\muk$ and $K^{-\ve'}$. 
\end{enumerate}
\end{remarks}

Of course there is the symmetric situation, where the coordinates are interchanged, and $A$ is replaced by $B$, which will give us a similar set of rectangles $\mathcal{L}_2$. 

 To cover also the Cases A and B from Section \ref{notstrong}, we
furthermore need caps of size $\sim\mu^{1/2}K^{-1/4}$.
Let $\mathcal{L}_3$ be a collection of squares  of side length $C\mu^{1/2}K^{-1/4}$ covering $\Sigma$, whose centers are  $\frac C2\mu^{1/2}K^{-1/4}$ separated. 

\smallskip
Finally, we set $\mathcal{L}:=\mathcal{L}_1\cup\mathcal{L}_2\cup\mathcal{L}_3$. This family is already well suited for proving our geometric Lemma \ref{geometric}. However, later on the application of  polynomial partitioning method will  even require a family  which is closed under intersections. We therefore define
$$\bar{\mathcal{L}}:=\{L_1\cap\ldots\cap L_m:L_1,\ldots,L_m\in \mathcal{L}_1\cup\mathcal{L}_2\cup\mathcal{L}_3,\ m\in\N\}.
$$ 
Note that due to Remark \ref{famprop}(ii), the number $m$ of possible intersections is uniformly bounded, and
each $\Delta\in \bar{\mathcal{L}}$  is either contained in a strip $L\in\mathcal{L}_1\cup\mathcal{L}_2$ of dimensions $\mu^{1/2}K^{-3/4}\times b$ (respectively $b\times\mu^{1/2}K^{-3/4}$) for some $b$ with  $\mu^{1/2}K^{-3/4}\le b\le K^{-\ve'},$ or is contained in a large cap $L\in\mathcal{L}_3$ of side length $C\mu^{1/2} K^{-1/4}.$
It is then easy to verify the following
\begin{remark}\label{fambarprop}
There exists an absolute constant $\bar N\in\N$ such that at most $\bar N$ sets $\Delta\in\bar{\mathcal{L}}$ overlap at any given point  $z\in\Sigma.$
\end{remark}

\smallskip

\subsection{The geometric lemma}
A key observation  which is important to adapt the polynomial partitioning method is that  caps which are mutually not strongly transversal are somewhat sparse.

\medskip
 Let us fix a parameter   $\ve'>0$ depending on $\ve$ and $\bar N$ and  then  $M=M(\ve)\in\N$ so that 
\begin{equation}\label{epsprime}
10\bar N\ve' =\ve^8 \quad \text{and}\quad   \frac3{4M(\ve)} \le \ve'\le\frac3{2M(\ve)},
\end{equation}
where $\bar N$ is the constant from Remark \ref{fambarprop}.

\begin{lemma}{\bf(The Geometric Lemma)}\label{geometric}
There is some constant $K_1(\epsilon)$ such that, for every $K\ge K_1(\epsilon)$ and every  $\phi\in \Hyp^{M(\ve)},$ the following holds true:  if  $\F$ is any family  of caps of side length $\mu^{1/2}K^{-1}$ which does not contain two strongly separated caps, then there is subcollection $\mathcal L_0\subset\mathcal L$ of cardinality $|\mathcal L_0|\le  K^{3\epsilon'},$ such that each cap in $\F$  is contained in   at least one element of $\mathcal L_0.$
 \end{lemma}

\begin{proof}  
Fix a cap $\tau_1\in\F,$ and recall that we denoted its center by $z_1^c=(x_1^c,y_1^c)$. If $\tau_2\in\F$, then
$$
\min\{|y^c_2-y^c_1|,|t^2_{z^c_1}(z^c_1,z^c_2)|\}\le  100 \mu^{1/2}K^{-1}\text{ and } \min\{|y^c_2-y^c_2|,|t^2_{z^c_2}(z^c_1,z^c_2)|\}\le  100 \mu^{1/2}K^{-1}.
$$
We will follow the discussion in Section \ref{notstrong}, and the division into cases we devised there. In cases A and B, both caps $\tau_1,\tau_2$ are clearly contained in a cap of size $100 \mu^{1/2}K^{-1/4},$ which is contained in a cap of size $C\mu^{1/2}K^{-1/4}$ from our collection $\mathcal{L}_3,$  if we assume that $C$ is sufficiently large.
\smallskip

 This leaves us with {\bf  Case C,} where
\begin{equation}\label{friday}
|t^2_{z^c_1}(z^c_1,z^c_2)|\}\le  100 \mu^{1/2}K^{-1} \ \text{and} \ |A(z^c_1)-A(z^c_2)|\le \mu^{1/2}K^{-3/4}.
\end{equation}
 
 Recall from  \eqref{nablat2} that $\partial_{x} t^2_{z^c_1}(z^c_1,z)\geq 1/2.$ In Lemma \ref{curvesI} we used   \eqref{nablat2} in order to show that we may parametrize the zero set of $t^2_{z^c_1}(z^c_1,\cdot)$ by a curve $\gamma=(\gamma_1,\gamma_2):[-1,1]\to\Sigma$ such that $\ga_2(t)=t.$ 
 In particular,  $\gamma_2(y_2^c)=y_2^c$. But then $z_2^c-\gamma(y_2^c)=(x_2^c-\gamma_1(y_2^c),0),$ and thus
\begin{equation}\label{thursday}
	|z_2^c-\gamma(y_2^c)|
	\leq 2|t^2_{z^c_1}(z^c_1,z^c_2)-t^2_{z^c_1}(z^c_1,\gamma(y_2^c))|
	=2|t^2_{z^c_1}(z^c_1,z^c_2)|
	\leq 200\mu^{1/2}K^{-1}. 
\end{equation}

Consider the $C^{M}$ function 
$$g(t) := A(\gamma(t))-A(z_1^c), \qquad t\in [-1,1].$$

Applying Corollary \ref{sublevelset} to $g$ and level $\lambda:=2 \muk$,  by our choice of $M=M(\ve) $  there exists a family $\I$ of subintervals of $[-1,1]$ such that 
\begin{enumerate}
\item    $|\I|\le 60  M \big(1+\|g^{(M)}\|^{1/M}_\infty K^{3/4M}\big) \le G(M)\,K^{\ve'},$
\item for all $I\in\I$ and all $ t\in I$ we have $|g(t)|< 16\muk,$
\item and for all $t\in I_0\backslash\bigcup_{I\in\I} I$ we have $|g(t)|\geq 2 \muk,$
\end{enumerate}

where $G(M)\ge 1$  for any integer $M\ge 3$  is a constant such that the last estimate in (i) holds uniformly for any $\phi\in \Hyp^M.$ Indeed,  by our definition of the function $g$ and the uniform estimates that are assumed to hold for all functions  $\phi$ in $\Hyp^M$ in combination with  Fa\`a di Bruno's theorem \cite{Sp05}  on derivatives of compositions of functions we easily see that  a uniform estimate 
$$
60  M \big(1+\|g^{(M)}\|_\infty^{1/M})\le G(M)
$$
holds true for all  $\phi\in \Hyp^M.$

In particular, if we set $K_1(\ve):= G(M(\ve))^{8M(\ve)/3},$ then in combination with \eqref{epsprime} we see that 
\begin{enumerate}
\item[(i')] $|\I|\le K^{3\ve'/2}$ if $ K\ge K_1(\ve).$
\end{enumerate}

By \eqref{nablaA}, \eqref{friday} and \eqref{thursday}, we see that 
$$|g(y_2^c)|\leq|A(\gamma(y_2^c))-A(z_2^c)|+|A(z_2^c)-A(z_1^c)|
<  2\muk,
$$
and hence by (iii), $y_2^c\in I$ for some $I\in\I$.  Again by \eqref{thursday}, we see that $\tau_2$ is contained in the neighborhood
$$\cN(I):=\gamma(I)+B(0,\muk)$$
of the curve $\gamma(I)$.

\medskip
Choose $k$ so that $|A_k-A(z^c_1)|\leq \frac{1}{2}C\muk$.
Fix any $t_0\in I,$ and choose  $S_{k,j}$ to be one of our strips of direction $\omega_k$ which contains $\gamma(t_0)$. Since these strips slightly overlap, we may and shall choose $j$ so that $\dist(\gamma(t_0),\partial S_{k,j})\geq \frac12 C'\muk$.

 Recall also that   by (ii), for any $t\in I$ we have 
\begin{equation}\nonumber
	|A(\gamma(t))-A(z_1^c)|= |g(t)| \leq 16\muk,
\end{equation}
hence $\gamma(t)$ is in the intersections of the regions  $R_I$ and  $R_{II}$ as defined in \eqref{curves} (modified by a harmless factor 16, which we will ignore).
 Using Lemma \ref{anglestable}, we see that the tangential of $\gamma$ on $I$ points essentially in direction $\omega_k$, and more precisely we obtain that 
$$|\langle\gamma(t)-\gamma(t_0),\tilde\omega_k^\perp\rangle|\lesssim \muk.$$
That means that $\cN(I)\subset  S_{k,j}$, provided we choose the constant $C'$ from the construction of these strips big enough. 

For $z\in\cN(I)$, we find a $t\in I$ with $|\gamma(t)-z|<\muk$, so that
$$|A(z)-A_k|\leq |A(z)-A(\gamma(t))|+|g(t)|+|A(z^c_1)-A_k|
\leq (1+16+ C/2 )\muk <C\muk,$$
i.e., $\cN(I)\subset R^k_{II},$  provided $C>34$.
This shows that $\cN(I)\subset S^1_{k,j}$, and, since $\cN(I)$ is clearly connected, even that 
$\cN(I)$ is contained in a connected component of $S^1_{k,j}$. 
Recall that we had divided $S_{k,j}^1$ into rectangles $\tilde L_{k,j}^i$ of lengths at most $\frac12 K^{-\ve'},$ by artificially chopping any connected component that is too long.
 But, for any $\tau_2\subset\cN(I)$ with $\tau_2\cap\tilde L_{k,j}^i\neq\emptyset$, $\tau_2\subset L_{k,j}^i\in\mathcal{L}_1$, and there are at most $6K^{\ve'}$ sets $\tilde L_{k,j}^i$ in any connected component of $S^1_{k,j}$. In combination with (i'), this proves the  claim of the lemma also in Case C.
\end{proof}

\begin{figure}[!h]
\centering
\includegraphics[scale=0.3]{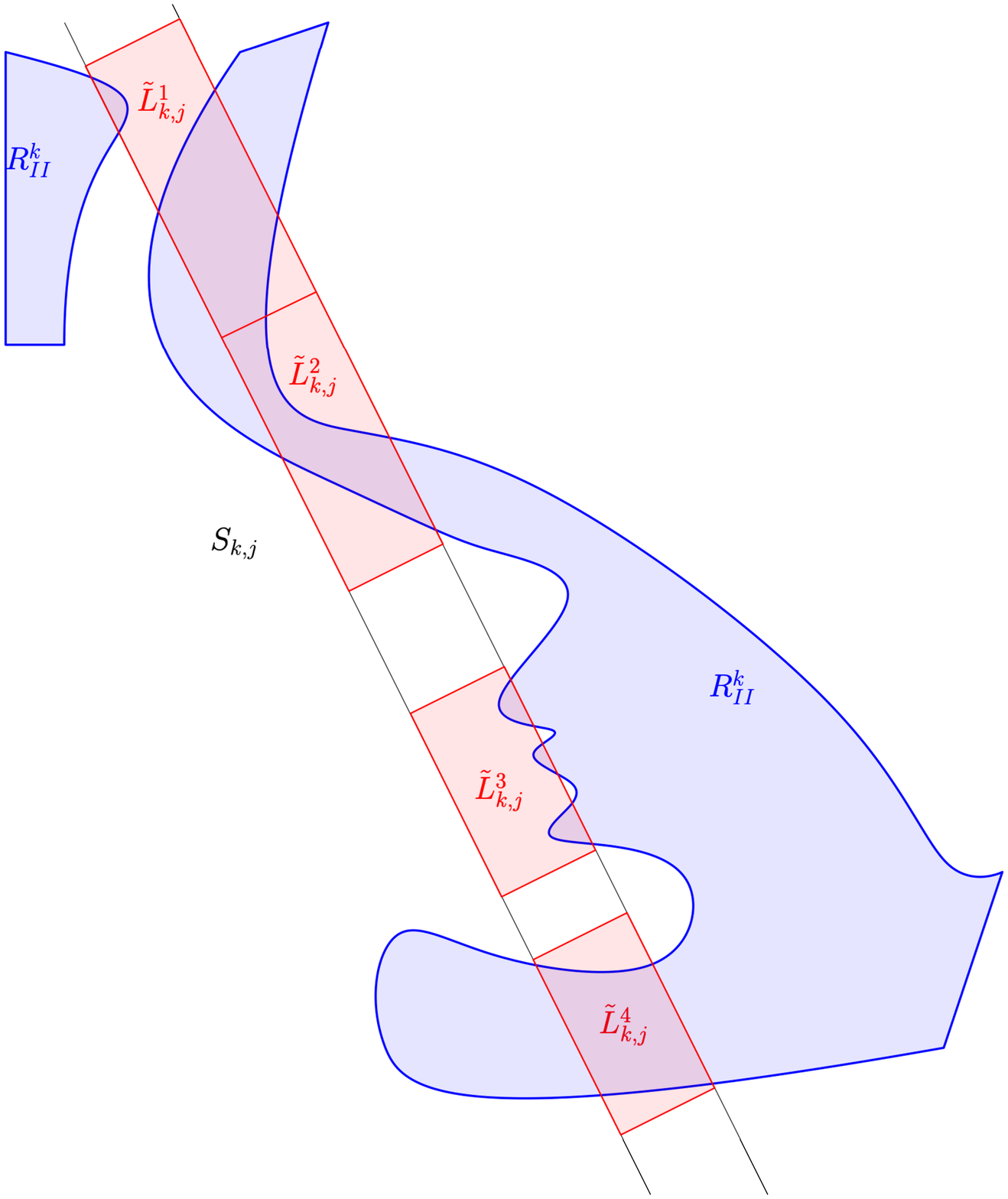}
\end{figure}


\setcounter{equation}{0}
\section{Reduction to estimates for the broad part  }\label{reductiontobroad}

 In this Section, we fix $\mu=1$ and  consider our basic decomposition of $\Sigma$ into caps $\tau$ of side length $K^{-1}$ which have pairwise disjoint interiors. Consider  the families $\mathcal L$  and $\overline{\mathcal L}$  defined in the previous section for $\mu=1.$ Recall that each $\Delta\in \overline{\mathcal L}$ is either contained in a strip  of dimensions $K^{-3/4}\times b$ for some $K^{-3/4}\le b\le K^{-\epsilon'},$ or is contained in a large cap of side length $100 K^{-1/4}.$

\smallskip

As in the previous section (for $\mu=1$),
given the function $f,$ $\alpha\in(0,1)$ and $K,$ we say that the point $\xi\in\R^3$ is {\it $\alpha$-broad for
$\E f$} if
$$
 \max_{\Delta\in \overline{\mathcal L} }|\E f_\Delta(\xi)|\le\alpha|\E f(\xi)|,
$$
where  $f_\Delta:=\sum\limits_{\tau\subset \Delta} f_\tau.$

We define $Br_\alpha\E f(\xi)$ to be $|\E f(\xi)|$ if $\xi$ is $\alpha$-broad, and zero
otherwise.

\medskip

We shall prove the following analogue to \cite[Theorem 2.4]{Gu16} and   \cite[Theorem 2.1]{bmvp20b}:

\smallskip

\begin{thm}\label{broadtheorem}
 Let $0<\epsilon< 10^{-10},$  and choose $M=M(\epsilon)\in \N$ sufficiently large so that \eqref{epsprime}  is satisfied. Then there are constants $K=K(\epsilon)\gg 1$ and $C_\epsilon$ such that for any $\phi\in  \Hyp^M$ and 
any radius $R\ge1$ 	 the following hold true:
$$
		\|Br_{K^{-\epsilon}}
\E f\|_{L^{3.25}(B_R)} \leq C_\epsilon R^\epsilon
\|f\|_{L^2(\Sigma )}^{12/13}\, \|f\|_{L^\infty(\Sigma )}^{1/13}
$$
for every $f\in L^\infty(\Sigma ),$ and moreover $K(\epsilon)\rightarrow\infty$ as
$\epsilon\rightarrow0.$
\end{thm}

To show that Theorem \ref{mainresult} follows from this result, let us note that it is enough to consider $q$ close to $2.6$ and put  $p:=3.25.$ We  divide the domain of integration $B_R$ in  \eqref{mainresultest}   into   three subsets:
\begin{eqnarray*}
A&:=&\{\xi\in B_R:\xi\text{ is $K^{-\epsilon}$-broad for $\E f$}\},\\
B&:=&\{\xi\in B_R:|\E f_\Delta(\xi)|> K^{-\epsilon}|\E f(\xi)|\text{ for some }\Delta\in \overline{\mathcal L},\\
&&\hskip6cm \,\text{with } \Delta \text{ contained in a strip } L\in\mathcal L_1\cup\mathcal{L}_2\}, \\
C&:=&\{\xi\in B_R\setminus B:|\E f_\Delta(\xi)|> K^{-\epsilon}|\E f(\xi)|\text{ for some }\Delta\in \overline{\mathcal L}, \\
&&\hskip6.2cm  \text{with } \Delta \text{ contained in a large cap } L\in\mathcal L_3 \}.
\end{eqnarray*}

If $\xi\in A$, then $|\E f(\xi)|=Br_{K^{-\epsilon}}\E f(\xi)$, so that the contribution of $A$ can be controlled using Theorem \ref{broadtheorem}. Notice that
$$\|f\|_{L^2(\Sigma )}^{12/13}\, \|f\|_{L^\infty(\Sigma )}^{1/13}\leq\|f\|_{L^2(\Sigma )}^{2/q}\, \|f\|_{L^\infty(\Sigma )}^{1-2/q},$$
since $q>2.6>13/6$.
\smallskip

 If $\xi\in B,$ then  there is some $\Delta\in\overline{\mathcal L}$ which is contained in a strip of dimensions $K^{-3/4}\times b$ for some $b\le K^{-\epsilon'},$ so that $|\E f_\Delta(\xi)|> K^{-\epsilon}|\E f(\xi)|.$  Then we may estimate
$$
|\E f(\xi)|< K^\epsilon \sup_\Delta|\E f_{\Delta}(\xi)| \leq K^\epsilon \Big(\sum_\Delta|\E f_\Delta(\xi)|^p\Big)^{1/p},
$$
where here the supremum and sum are taken over all  $\Delta\in \bar{\mathcal{L}}$ which are contained  in a strip $L\in\mathcal L$ of dimensions $K^{-3/4}\times b$ for some $b\le K^{-\epsilon'}$ (which may depend on $\Delta$). Thus, we can apply to any such $f_\Delta$ the scaling (associated to the corresponding strip $L$) described in Subsection \ref{indu},  more precisely estimate \eqref{indest}, and obtain
$$
 \|\E f_\Delta\|_{L^p(B_R)}\le C' _\ve  R^\epsilon (K^{-3/4})^{\frac 1{q'}-\frac 2p+\ve}\|f_\Delta\|_2^{2/q}\|f\|_\infty^{1-2/q}. $$
 
 Therefore,
 \begin{eqnarray*}
	\|\E f\|_{L^p(B)}
	&\leq&	K^{\epsilon} \Big(\sum_\Delta\|\E f_\Delta\|^p_{L^p}\Big)^{1/p}\\
	&\leq& C'_\ve K^{\epsilon} R^\epsilon (K^{-3/4})^{\frac 1{q'}-\frac 2p+\ve}
\Big(\sum_\Delta\|f_\Delta\|_{2}^{2p/q}\,\|f\|_{\infty}^{p(1-2/q)}\Big)^{1/p}.
\end{eqnarray*}
Since $2p/q>2,$ taking into account the overlap of the elements of $\overline{\mathcal L}$ (see Remark \ref{fambarprop}) we estimate
$$
\sum_\Delta\|f_\Delta\|_{2}^{2p/q}\le \Big(\sum_\Delta\|f_\Delta\|_{2}^{2}\big)^{p/q}\le\bar N^{p/q}\|f\|_{2}^{2p/q}.
$$
Hence,
$$
	\|\E f\|_{L^p(B)}\leq  C'_\ve \bar N^{1/q} K^{\epsilon} R^\epsilon (K^{-3/4})^{\frac 1{q'}-\frac 2p+\ve}\|f\|_{2}^{2/q}\,\|f\|_{\infty}^{(1-2/q)}
$$

$$
\leq \frac 1{10}C_{\epsilon} R^\epsilon \|f\|_{2}^{2/q}\,\|f\|_{\infty}^{1-2/q},
$$
since $p> 2q'$.

\medskip

For $\xi\in C,$ isotropic scaling gives the same result.


\setcounter{equation}{0}
\section{Proof of Theorem \ref{mainresult2}}\label{proofthm}

Following Section 3 in \cite{Gu16} and \cite{bmvp20b}, we shall next formulate a more general statement in Theorem \ref{broadtheorem} which will become amenable to inductive arguments. We have to consider $\mu\ge1.$

\medskip

We assume that we are given $\mu\ge 1,$ a dyadic natural number  $K\gg 1$ and a family of caps  $\tau$ of side length $\mu^{1/2}K^{-1},$ covering $\Sigma =[0,1]\times[0,1],$ such that their centers are $K^{-1}$-separated. Hence,  at any point there will be at most  $\mu$  of these caps which overlap at that point. Notice also that there are at most $K^2$
 caps $\tau$  in the family. We also assume that we have a decomposition
\begin{equation}\nonumber
 f=\sum_\tau f_\tau,
\end{equation}
where $\supp f_\tau\subset\tau.$
\smallskip

We adapt the notion of broadness to the modified family of caps $\tau$:
For each $\Delta\in\bar{\mathcal{L}},$ 
we define $f_\Delta:=\sum\limits_{\tau\subset \Delta} f_\tau.$

\smallskip

Let $\alpha\in(0,1).$
Given the function $f$  and $K,$ we say that the point $\xi\in\R^3$ is {\it $\alpha$-broad for}
$\E f$ if
$$
 \max_{\Delta\in \overline{\mathcal L}} |\E f_\Delta(\xi)|\le\alpha|\E f(\xi)|.
$$

We define $Br_\alpha\E f(\xi)$ to be $|\E f(\xi)|$ if $\xi$ is $\alpha$-broad, and zero
otherwise.

\medskip

Theorem \ref{broadtheorem} will be a consequence of the following

\begin{thm}\label{largetheorem}
Let $0<\epsilon< 10^{-10}.$  
Then there are constants $M=M(\epsilon)\in \N, \,K=K(\epsilon)\gg 1,$ and $C_\epsilon$ such that for any $\phi\in  \Hyp^M,$   any family of caps $\tau$  with multiplicity at most $\mu$ covering $\Sigma$ as above  and associated  functions $f_\tau$  which decompose $f,$  
any radius $R\ge 1,$ and any $\alpha\ge K^{-\epsilon},$  the following
hold true:
\smallskip

If for every $\omega\in \Sigma,$ and every cap $\tau$  as above,
\begin{equation}\label{average}
\oint_{B(\omega,R^{-1/2})}|f_\tau|^2\le 1,
\end{equation}
then
\begin{equation}\label{broadest}
\int_{B_R}(Br_{\alpha} \E f)^{3.25} \leq C_\epsilon R^\epsilon
\bigg(\sum_\tau\int|f_\tau|^2\bigg)^{3/2+\epsilon}R^{\delta_{trans}\log(K^\epsilon\alpha\mu)},
\end{equation}
where $\delta_{trans}:=\epsilon^6.$ Moreover,
$K(\epsilon)\rightarrow\infty$ as $\epsilon\rightarrow0.$
\end{thm}

Here, in $\R^n,$ by $B(\omega,r)$ we denote the Euclidean ball of radius $r>0$ and center $\omega,$ and by $\oint_A f:=\frac 1{|A|}\int_A f$ we
denote the mean value  of $f$ over the measurable set $A$ of volume $|A|>0.$

\medskip

From this, as in \cite{bmvp20b} and \cite{Gu16}, Theorem \ref{broadtheorem} follows taking $\mu=1.$ 

\medskip

As in  these papers, let us also choose 
\begin{equation}\label{deltas}
\delta_{trans}:=\epsilon^6,\qquad \delta_{deg}:=\epsilon^4, \qquad \delta:=\epsilon^2,
\end{equation}
  so that in particular
$$
\delta_{trans}\ll\delta_{deg}\ll \delta\ll \epsilon< 10^{-10}.
$$  
 We next choose $M=M(\epsilon)\in \N$ and $\ve'>0$  according to  \eqref{epsprime}, i.e.,   
\begin{equation*}
10\bar N\ve' =\ve^8 \quad \text{and}\quad   \frac3{4M(\ve)} \le \ve'\le\frac3{2M(\ve)}, 
\end{equation*}
and assume  that $\ve>0$ is so small that  also 
 $M\delta\gg 1000$  holds true (compare Proposition \ref{packets}).

\smallskip
 
We also set
\begin{equation}\label{KD}
K=K(\epsilon):= \floor{K_1(\ve) +e^{\epsilon^{-10}}+1},  \qquad D=D(\epsilon):=R^{\delta_{deg}}=R^{\epsilon^4},
\end{equation}
where $K_1(\ve)$ is the constant from the Geometric Lemma \ref{geometric} and $\floor{x}$ denotes  the integer part of $x,$ and where we are assuming that $R\ge 1$ is  given. 
In particular, we then have $K\ge K_1(\ve),$ as assumed in the Geometric Lemma.

\smallskip

The following remark  is  similar to Remark 4.1 in \cite{bmvp20b} and holds (with the same proof) also here:

\begin{remarks}\label{rem4.1}
a) It is enough to consider the case where $\alpha\mu\le K^{-\epsilon/2},$  because in the
other case, the exponent $\delta_{trans}\log(K^\epsilon\alpha\mu)$ is very large  and the estimate \eqref{broadest} trivially holds true.   Note that since in Theorem \ref{largetheorem} we are also assuming that  $\alpha\ge K^{-\epsilon},$
this implies that $\mu\le K^{\ve/2}. $
Henceforth, we shall therefore always assume that $\alpha\mu\le K^{-\epsilon/2}\le 10^{-5}.$ 

b) It is then also enough to consider the case where  $R$ is tremendously  bigger than $K,$ say $R\ge 1000\, e^{K^{ \,e^{\epsilon^{-1000}}}}.$
\end{remarks}

As usual, we will work with wave packet decompositions of the functions $f$ defined on
${\bf S_\phi}:=\{(x,y,\phi(x,y)\;:\; (x,y)\in \Sigma\}.$ Following
\cite{Gu16}, we decompose $\Sigma $ into  squares (``caps'') $\theta$ of side length  $R^{-1/2}.$  By
$\omega_\theta$ we shall denote the center of $\theta,$ and by  $\nu(\theta)$  the ``outer'' unit normal  to
${\bf S_\phi}$ at the point $(\omega_\theta,\phi(\omega_\theta))\in {S_\phi},$ which points into the direction of $(-\nabla \phi(\omega_\theta),-1).$
 $\T(\theta)$ will denote a set  of  $R^{1/2}$-separated tubes $T$ of radius $R^{1/2+\delta}$ and length $R,$ which are  all parallel to  $\nu(\theta)$ and for which the corresponding thinner tubes of radius $R^{1/2}$ with the same axes
cover  $B_R.$ We will  write $\nu(T):=\nu(\theta)$ when $T\in\T(\theta).$

Note that  for each $\theta,$ every  point $\xi\in B_R$ lies in
$O(R^{2\delta})$ tubes $T\in\T(\theta).$ We put  $\T:=\bigcup\limits_{\theta} \T(\theta).$  Arguing in the same way as in \cite[Proposition 2.6]{Gu16},  and observing that it is  not necessary for the arguments in the proof, which are based on integrations by parts,  that the phase $\phi$ is $C^\infty$, but merely  $C^{M}$ for sufficiently large $M$ (more precisely,  $M\delta\gg 1000$),  we arrive at the following  approximate wave packet decomposition.  
\begin{proposition}\label{packets}
Assume that $R$  is  sufficiently large (depending on $\delta$).  Then, for any $\phi\in\Hyp^M$   (with $M=M(\ve)$ as before), given  $f\in
L^2(\Sigma),$ we may associate to  each tube $T\in\T$  a function $f_T$ such that the following hold true:
\begin{itemize}
\item[ a)] If $T\in\T(\theta),$ then $\supp f_T\subset 3\theta.$
\item[ b)] If $\xi\in B_R\setminus T,$ then $|\E f_T(\xi)|\le R^{-1000}\|f\|_2.$
\item[c)] For any $x\in B_R,$ we have $|\E f(x)-\sum_{T\in\T}\E f_T(x)|\le
    R^{-1000}\|f\|_2.$
\item[d)] (Essential orthogonality) If $T_1,T_2\in \T(\theta)$ are disjoint, then \hfill \newline $\big|\int f_{T_1} \overline{f_{T_2} }\big | \le R^{-1000} \int_{3\theta} |f|^2.$
\item[e)] $\sum_{T\in\T(\theta)}\int_{\Sigma}|f_T|^2\le C\int_{3\theta}|f|^2.$

\end{itemize}
\end{proposition}

\smallskip

We next recall the version of the polynomial ham sandwich theorem with non-singular polynomials from \cite{Gu16}.
 If $P$ is a real polynomial on $\R^n,$ we denote by $Z(P):=\{\xi\in\R^n: P(\xi)=0\}$ its  null variety. $P$ is said to be {\it non-singular} if $\nabla P(\xi)\ne 0$ for every point $\xi\in Z(P).$

Then, by  in \cite[Corollary 1.7]{Gu16}, there is a non-zero polynomial $P$ of degree at most $D$ which is a product of non-singular polynomials
such that the set $\R^3\setminus Z(P)$ is a disjoint union of $\sim D^3$ cells $O_i$ such that,
for every $i,$
\begin{equation}\label{D}
\int_{O_i\cap B_R} (Br_{\alpha} \E f)^{3.25}\sim D^{-3}\int_{B_R} (Br_{\alpha} \E f)^{3.25}.
\end{equation}

We next define $W$ as the $R^{1/2+\delta}$ neighborhood of $Z(P)$ and put $O_i':=(O_i\cap B_R)\setminus W.$

Moreover,  note that if we apply Proposition \ref{packets}  to $f_\tau$ in place of $f$ (what we shall usually do), then by property (a) in Proposition \ref{packets}, for every tube $T\in \T$ the function $f_{\tau,T}$ is supported in an $O(R^{-1/2})$ neighborhood of $\tau.$  Following Guth, we define
$$
\T_i:=\{T\in\T:  T\cap O_i'\ne\emptyset\},\quad f_{\tau,i}:=\sum_{T\in\T_i}
f_{\tau,T},\quad f_{\Delta,i}:=\sum_{\tau\subset \Delta} f_{\tau,i} \quad \text{and}\quad f_i:=\sum_\tau f_{\tau,i}.
$$
Then we can use the following analogue to  \cite[Lemma 3.2]{Gu16}:
\begin{lemma} \label{lemma3.2}
Each tube $T\in\T$ lies in at most $D+1$ of the sets $\T_i.$
\end{lemma}
We cover $B_R$ with $\sim R^{3\delta}$ balls $B_j$ of radius $R^{1-\delta}.$
Recall Definitions  3.3 and 3.4 from \cite{Gu16}:
\begin{defns} {\rm
a) We define $\T_{j,tang}$ as the set of  all tubes $T\in\T$ that satisfy the following conditions:
$$T\cap W\cap B_j\ne\emptyset,$$
and  if $\xi\in Z(P)$ is any  nonsingular point (i.e., $\nabla P(\xi)\ne 0$) lying in $2B_j\cap 10T,$ then
$$
{\rm angle}(\nu(T), T_\xi Z(P))\le R^{-1/2+2\delta}.
$$
Here, $T_\xi Z(P)$ denotes the tangent space to $Z(P)$ at $\xi,$ and we recall that $\nu(T)$ denotes the unit vector in direction of $T.$
Accordingly, we define
$$ f_{\tau,j,tang}:=\sum_{T\in\T_{j,tang}}f_{\tau,T}\quad\text{and}\quad
f_{j,tang}:=\sum_\tau
f_{\tau,j,tang}.
$$
b) We define $\T_{j,trans}$ as the set of  all tubes $T\in\T$ that satisfy the following conditions:
$$T\cap W\cap B_j\ne\emptyset,
$$
and there exists a nonsingular point $\zeta\in Z(P)$ lying in $2B_j\cap 10T,$ so that
$$
{\rm angle}(\nu(T), T_\zeta Z(P))> R^{-1/2+2\delta}.
$$
Accordingly, we define
$$ f_{\tau,j,trans}:=\sum_{T\in\T_{j,trans}}f_{\tau,T}\quad\text{and}\quad
f_{j,trans}:=\sum_\tau f_{\tau,j,trans}.
$$}
\end{defns}
We also recall Lemmas 3.5 and 3.6 in \cite{Gu16}:
\begin{lemma}\label{lemma3.5}
Each tube $T\in\T$ belongs to at most ${\rm Poly}(D)=R^{O(\delta_{deg})}$ different sets
$\T_{j,trans}.$
\end{lemma}
\begin{lemma}\label{lemma3.6}
For each $j,$ the number  of different $\theta$ so that
$\T_{j,tang}\cap\T(\theta)\ne\emptyset$ is at most $R^{1/2+O(\delta)}.$
\end{lemma}

Note that the previous lemma makes use of the fact that the Gaussian curvature does not vanish on the surface $\Sigma,$ so that the Gau\ss{} map is a diffeomorphism onto its image.

\begin{lemma}\label{lemma3.7}
If $\xi \in O_i'.$ Then, given  our assumptions on $R$  from Remarks \ref{rem4.1}, we have
$$
Br_\alpha \E f(\xi)\le  Br_{2\alpha} \E f_i(\xi)+R^{-900}\sum_\tau\|f_\tau\|_2.
$$
\end{lemma}

\begin{proof}
This is analogous to \cite[Lemma 4.7]{bmvp20b}. The proof of that lemma, without any changes, also gives us Lemma \ref{lemma3.7}.
\end{proof}

\medskip

The remaining part  of this subsection will be  devoted to the proof of the following crucial analogue to  the key Lemma 3.8 in \cite{Gu16}:
\begin{lemma}\label{lemma3.8}
If $\xi\in B_j\cap W$ and  $\alpha\mu\le 10^{-5},$ then
\begin{equation}\label{crucialbroad}
Br_\alpha\E f(\xi)\le 2\bigg(\sum_IBr_{K^{4\bar N\epsilon'}\alpha}\E f_{I,j,trans}(\xi)+K^{100} {\rm Bil}(\E
f_{j,tang})(\xi)+R^{-900}\sum_\tau\|f_\tau\|_2\bigg),
\end{equation}
where the first sum is over all possible subsets $I$ of the given family of caps $\tau.$
\end{lemma}

\begin{proof}

Let $\xi\in B_j\cap W. $ We may assume that $\xi$ is $\alpha$-broad for $\E f$ and that  $|\E
f(\xi)|\ge
R^{-900}\sum_\tau\|f\|_2.$ Let
\begin{equation}\label{I}
I:=\{\tau: |\E f_{\tau,j,tang}(\xi)|\le K^{-100} |\E f(\xi)|\}.
\end{equation}
We consider two possible cases:
\smallskip

\noindent{\bf Case 1:}  $I^c$ contains two strongly separated caps $\tau_1$ and
$\tau_2.$ Then trivially
\begin{equation}\nonumber
|\E f(\xi)|\le K^{100}|\E f_{\tau_1,j,tang}(\xi)|^{1/2}|\E f_{\tau_2,j,tang}(\xi)|^{1/2}\le
K^{100}{\rm Bil}(\E f_{j,tan})(\xi),
\end{equation}
hence \eqref{crucialbroad}.

\smallskip
\noindent{\bf Case 2:}  $I^c$ does not contain two strongly separated caps.
\smallskip

We  denote by $\mathcal L(\xi)\subset \mathcal L$ the family of at most  $ K^{3\epsilon'}$ strips  respectively large caps given by the Geometric Lemma  for the family of caps $\F:=I^c.$ 
By
$$
J:=\{\tau\,:\; \tau\subset L \text{ for some } L\in\mathcal L(\xi)\}
$$
we denote     the corresponding subset of caps $\tau.$ Then $I^c\subset J,$ i.e.,  $J^c\subset I.$ We write
\begin{equation}\nonumber
	f=\sum_{\tau\in J} f_\tau+\sum_{\tau\in J^c}f_\tau.
\end{equation}
Hence,
$$
|\E f(\xi)|\leq |\sum_{\tau\in J}\E f_{\tau}(\xi)|+|\sum_{\tau\in J^c}\E f_\tau(\xi)|.
$$
For $L\in\mathcal L,$ we denote by $\tilde L:=\{\tau\,:\;\tau\subset L\}.$ Note that $f_L=\sum_{\tau\in\tilde L}f_\tau,$
and
$$
J=\bigcup_{L\in\mathcal L(\xi)} \tilde L.
$$
Thus, by the inclusion-exclusion principle and Remark \ref{fambarprop},
\begin{equation}\label{inclusexcluse}
\chi_J=\sum_{k=1}^{\overline N}(-1)^{k+1}\sum_{L_1,\dots,L_k\in \mathcal L(\xi)}\chi_{\tilde L_1\cap\cdots\cap\tilde L_k}.
\end{equation}
Therefore,
$$
|\sum_{\tau\in J}\E f_\tau(\xi)|\le \sum_{k=1}^{\bar N}\sum_{L_1,\dots,L_k\in \mathcal L(\xi)}|\sum_{\tau\in \tilde L_1\cap\cdots\cap\tilde L_k}\E f_\tau(\xi)|=\sum_{k=1}^{\bar N}\sum_{L_1,\dots,L_k\in \mathcal L(\xi)}|\E f_{L_1\cap\cdots\cap L_k}(\xi)|.
$$
Since $\xi$ is $\alpha$-broad, and  since according to Remark \ref{rem4.1} we  may assume that $\alpha\le K^{-\epsilon/2},$  this can be further estimated by \
\begin{equation}\nonumber
\le\sum _{k=1}^{\bar N}|\mathcal L(\xi)|^k \alpha|\E f(\xi)|
\le \bar N (K^{3\ve'})^{\bar N} \alpha|\E f(\xi)|
\le \bar N K^{3\ve'\bar N} K^{-\ve/2} |\E f(\xi)|
\le\frac1{10}|\E f(\xi)|,
\end{equation}
 as one can easily see by our choices of  $K$ in \eqref{KD} and $\ve'$ in \eqref{epsprime}, provided $\ve>0$ is assumed to be  sufficiently small.
Thus,
$$
|\E f(\xi)|\le \frac1{10}|\E f(\xi)|+|\sum_{\tau\in J^c}\E f_\tau(\xi)|,
$$
and therefore
$$
|\E f(\xi)|\le \frac{10}9 |\sum_{\tau\in J^c}\E f_\tau(\xi)|.
$$
Since $\xi\in B_j\cap W,$ by Proposition \ref{packets},
\begin{equation}\label{4.21n}
	\E f_\tau(\xi)=\E f_{\tau,j,trans}(\xi)+\E f_{\tau,j,tang}(\xi)+ O(R^{-1000})\|f_\tau\|_2.
\end{equation}

Moreover, since $J^c\subset I,$ and  since there are at most $K^2$ caps $\tau,$
\begin{equation}\label{tangesti}
\sum_{\tau\in J^c}|\E f_{\tau,j,tang}(\xi)|\le\sum_{\tau\in I}|\E f_{\tau,j,tang}(\xi)|\le
K^{-100}\sum_{\tau\in I}|\E f(\xi)|\le K^{-98}|\E f(\xi)|,
\end{equation}
where the second inequality is a consequence of the definition of $I.$
Thus,
\begin{eqnarray*}
\frac 9{10}|\E f(\xi)|&\le& |\sum_{\tau\in J^c}\E f_{\tau,j,trans}(\xi)|+K^{-98}|\E
f(\xi)|+\sum_\tau R^{-1000}\|f_\tau\|_2\\
&=&|\E f_{J^c,j,trans}(\xi)|+K^{-98}|\E f(\xi)|+\sum_\tau R^{-1000}\|f_\tau\|_2,
\end{eqnarray*}
and hence, since $|\E f(\xi)|\ge R^{-900}\sum\|f_\tau\|_2,$
\begin{equation}\label{11-9}
|\E f(\xi)|\le\frac{11}9|\E f_{J^c,j,trans}(\xi)|.
\end{equation}
\medskip
It will then finally suffice to show  that $\xi$ is $K^{4\bar N\epsilon'}\alpha$-broad  for $\E g,$ where  $g:= f_{J^c,j,trans}.$
To this end let us set $g_\tau:=f_{\tau,j,trans},$ if $\tau\in J^c,$ and zero otherwise,  so that 
$$g=\sum g_\tau.$$

In what follows, we shall use  the short hand notation ``$\rm neglig$''  for terms which are much smaller than
$R^{-940}\sum_\tau\|f_\tau\|_2.$ 

Observe first that  by \eqref{4.21n}
$$
|\E f_{\tau,j,trans}(\xi)|\le|\E f_\tau(\xi)|+|\E f_{\tau,j,tang}(\xi)|+{\rm neglig},
$$
so that  if $\tau\in J^c\subset I,$ then by the definition of $I,$
\begin{equation}\nonumber
|\E f_{\tau,j,trans}(\xi)|\le|\E f_\tau(\xi)|+K^{-100}|\E f(\xi)|+{\rm neglig}.
\end{equation}

\medskip

We have to show that
$$
|\E g_{\Delta}(\xi)|\le  K^{4\bar N\epsilon'}\alpha |\E g(\xi)|
$$
for all  $\Delta\in\overline{\mathcal L}.$  Write $\Delta=L_1\cap\cdots\cap L_r,$ where $L_i\in \mathcal L,$ and set $\tilde \Delta:=\tilde L_1\cap\cdots\cap\tilde L_r,$ so that $g_\Delta=\sum_{\tau\in \tilde \Delta}g_\tau=\sum_{\tau\in \tilde \Delta\cap J^c}f_{\tau,j,trans}.$ Therefore the following two cases can arise:
\smallskip

\begin{enumerate}
	\item There is some $i=1,\dots,r,$ such that $L_i\in \mathcal L(\xi).$ Then, $\tilde \Delta\subset\tilde L_i\subset J,$ hence, $\tilde \Delta\cap J^c=\emptyset.$
	\item For all $i=1,\dots,r,$ $L_i\notin \mathcal L(\xi).$
\end{enumerate}

Observe first that by summing \eqref{4.21n} over all $\tau\in\tilde \Delta\cap J^c$ we obtain
\begin{equation}\label{ongdelta}
|\E g_{\Delta}(\xi)|=|\sum_{\tau\in\tilde \Delta}\E g_\tau|\leq|\sum_{\tau\in \tilde \Delta\cap J^c}\E f_{\tau}(\xi)|+\sum_{\tau\in \tilde \Delta\cap J^c}|\E f_{\tau,j,tang}(\xi)|+{\rm neglig}.
\end{equation}
	
By \eqref{tangesti}, the second term can again be estimated  by
$$
	\sum_{\tau\in \tilde \Delta \cap J^c}|\E f_{\tau,j,tang}(\xi)|
	\leq K^{-98}|\E f(\xi)|.
$$
Case (i) is thus trivial.  In case (ii), we write
$$
\sum_{\tau\in \tilde \Delta\cap J^c}\E f_{\tau}(\xi)=\E f_{\Delta}(\xi)-\sum_{\tau\in \tilde \Delta\cap J} \E f_{\tau}(\xi).
$$
The first term is estimated using broadness. For the second term,  again by \eqref{inclusexcluse},

\begin{eqnarray*}
|\sum_{\tau\in\tilde \Delta\cap J}\E f_\tau(\xi)|&\le&
\sum_{k=1}^{\bar N}\sum_{L'_1,\dots,L'_k\in \mathcal L(\xi)}|\sum_{\tau\in \tilde \Delta\cap\tilde L'_1\cap\cdots\cap\tilde L'_k}\E f_\tau(\xi)| \\
&=&\sum_{k=1}^{\bar N}\sum_{L'_1,\dots,L'_k\in \mathcal L(\xi)}|\E f_{\Delta\cap L'_1\cap\cdots\cap L'_k}(\xi)|.
\end{eqnarray*}

Note that $\Delta\cap L'_1\cap\cdots\cap L'_k\in\overline{\mathcal L},$ so that, since $\xi$ is $\alpha$-broad for $\E f,$
$$
|\sum_{\tau\in\tilde \Delta\cap J}\E f_\tau(\xi)|\le\sum _{k=1}^{\bar N}|\mathcal L(\xi)|^k \alpha|\E f(\xi)|
\le \bar NK^{3\epsilon'\bar N}\alpha|\E f(\xi)|.
$$
Since $\alpha\ge K^{-\epsilon}\gg 10 K^{-98},$  in combination with \eqref{ongdelta}, and in the last step with \eqref{11-9}, we conclude that
$$
|\E g_{\Delta}(\xi)| \le (\bar N K^{3\bar N\epsilon'}+1)\alpha |\E f(\xi)|+K^{-98}|\E f(\xi)|+{\rm neglig} \le  K^{4\bar N\epsilon'} \alpha |\E g(\xi)|
$$

This completes the proof of Lemma \ref{lemma3.8}.
\end{proof}
\medskip

The contribution by the bilinear term in \eqref{crucialbroad} will be controlled by means of the following analogue to  \cite[Proposition 4.13 ]{bmvp20b} (or  \cite[Proposition 3.9 ]{Gu16}):
\begin{proposition}\label{prop3.9} We have
$$
\int_{B_j\cap W}{\rm Bil}(\E f_{j,tang})^{3.25}\le C_\epsilon R^{O(\delta)+\epsilon/2}\bigg(\sum_\tau\int
|f_\tau|^2\bigg)^{3/2+\epsilon}.
$$
\end{proposition}

With Proposition  \ref{prop3.9} at hand, the rest  of the proof of Theorem \ref{largetheorem}, which we shall detail in Subsection \ref{completing},  will follow the arguments in Section 4.2 in \cite{bmvp20b} (which in return are an adaptation of the arguments in pages 396-398 of \cite{Gu16}).
\smallskip

The proof of Proposition  \ref{prop3.9} reduces to the following analogue to  \cite[Lemma 4.14]{bmvp20b} and  \cite[Lemma 3.10]{Gu16}. Suppose we have covered $B_j\cap W$ with a minimal number of cubes $Q$  of side length $R^{1/2},$   and denote by $\T_{j,tang,Q}$ the set of all tubes $T$ in $\T_{j,tang}$ such that $10T$ intersects $Q.$
\begin{lemma}\label{lemma4.11} Fix $j,$ i.e., a ball $B_j.$
If $\tau_1,\tau_2$ are strongly separated  caps, then for any of the cubes $Q$  we have
\begin{eqnarray*}
&&\int_Q|\E f_{\tau_1,j,tang}|^2|\E f_{\tau_2,j,tang}|^2\\
&&\phantom\qquad\le
R^{O(\delta)}R^{-1/2}
\big(\sum_{T_1\in\T_{j,tang,Q}}\|f_{\tau_1,T_1}\|_2^2\big)
\big(\sum_{T_2\in\T_{j,tang,Q}}
\|f_{\tau_2,T_2}\|_2^2\big)+{\rm neglig}.
\end{eqnarray*}
\end{lemma}

\begin{proof}
Using Remark \ref{Gammasize}, the proof of Lemma 4.15 in \cite{bmvp20b} can be repeated word by word, giving the result.
\end{proof}

\medskip

\subsection{Completing the  proof of Theorem \ref{largetheorem}}\label{completing}

If we compare with  Subsection 4.2 in \cite{bmvp20b}, which deals with the induction arguments with respect to the size of the radius $R,$ and the size  of 
$\sum_\tau\int |f_\tau|^2,$ we can see that Lemma \ref{lemma3.8}, which was the only result whose proof  required substantial new arguments compared to the corresponding result in  \cite{bmvp20b}, is  needed only for the discussion of  Case 2 in this subsection, i.e., the  case where  the dominating term in 
$$
\int_{B_R} (Br_\alpha\E f)^{3.25}=\sum_i\int_{B_R\cap O_i'} (Br_\alpha\E
f)^{3.25}+\int_{B_R\cap W} (Br_\alpha\E f)^{3.25}
$$
is the second term, the ``wall''  term. The estimation of this term could then be reduced to controlling 
$\sum_{j}\int_{B_j\cap W}\sum_ I (Br_{150\alpha}\E
f_{I,j,trans})^{3.25},$ which in view of our  Lemma \ref{lemma3.8} here has to be modified to
\begin{equation}\label{three}
\sum_{j}\int_{B_j\cap W}\sum_ I (Br_{K^{4\bar N\epsilon'}\alpha}\E
f_{I,j,trans})^{3.25}. 
\end{equation}

Dealing with this term by induction as in that paper, we arrive at
\begin{eqnarray*}
\int_{B_R}(Br_{K^{4\bar N\epsilon'}\alpha}E f)^{3.25}&\le&M_\epsilon C_\epsilon {\rm Poly}(D) R^{\epsilon(1-\delta)}\bigg(\sum_\tau\int |f_\tau|^2\bigg)^{3/2+\epsilon}R^{\delta_{trans}(1-\delta)\log(K^{4\bar N\ve'}K^\ve\alpha\mu)}\\
&\le&\Big(M_\epsilon {\rm Poly}(R^{\delta_{deg}}) R^{\delta_{trans}4\bar N\ve'\log(K)-\epsilon\delta}\Big)\\
&&\hskip 3cm\times C_\epsilon R^{\epsilon}
\bigg(\sum_\tau\int|f_\tau|^2\bigg)^{3/2+\epsilon}
R^{\delta_{trans}\log(K^\epsilon\alpha\mu)}\\
&\le& C_\epsilon R^{\epsilon}
\bigg(\sum_\tau\int|f_\tau|^2\bigg)^{3/2+\epsilon}
R^{\delta_{trans}\log(K^\epsilon\alpha\mu)}.
\end{eqnarray*}
For the last inequality, note that the choices of $\delta,\delta_{deg},\delta_{trans}$  in \eqref{deltas} and $K$ in \eqref{KD},  in combination with \eqref{epsprime}, ensure that for $\ve$ sufficiently small
$$C\delta_{deg}+\delta_{trans}4\bar N\ve'\log(K)-\delta\epsilon
= C\ve^4+\ve^6 4 \bar N\ve'\ve^{-10} - \ve^3 <  -\ve^3/2.
$$
Here, $M_\ve=2^{K^2}.$ Then, by  Remark \ref{rem4.1}, $ M_\ve R^{-\ve^3/2} \ll 1.$ 
This completes the proof of Theorem 6.1.

\bigskip

\thispagestyle{empty}

\renewcommand{\refname}{References}

\end{document}